\documentclass[10pt]{amsart}
\usepackage{amsmath,amssymb,amsthm}
\usepackage[T1]{fontenc}
\usepackage{mathtools} \mathtoolsset{showonlyrefs} 

\usepackage{hyperref}
\usepackage[english]{babel}

\theoremstyle{plain}
  \newtheorem{theorem}[subsection]{Theorem}
  
  \newtheorem{lemma}[subsection]{Lemma}
  \newtheorem{corollary}[subsection]{Corollary}

\theoremstyle{definition}
  \newtheorem{definition}[subsection]{Definition}
  
  \newtheorem{assumptions}[subsection]{Assumptions}

\newcommand{\msL}{\mspace{-90mu}}
\allowdisplaybreaks

\newcommand{\R}{{\bf R}}
\newcommand{\N}{{\bf N}}
\newcommand{\C}{{\bf C}}
\newcommand{\id}{\mathrm{d}}

\newcommand{\wt}[1]{\widetilde{#1}}
\newcommand{\wh}[1]{\widehat{#1}}
\newcommand{\psd}{pseudo-differential }

\newcommand{\summ}[3]{\overline{\wt{#1}}_{#2}^{\mspace{2mu}{#3}}}
\newcommand{\ba}{\widetilde{\alpha}}
\newcommand{\bb}{\widetilde{\beta}}

\newcommand{\bx}{\widetilde{x}}
\newcommand{\by}{\widetilde{y}}

\newcommand{\bm}{\widetilde{m}} 
\newcommand{\be}{\widetilde{\eta}}
\newcommand{\bxi}{\widetilde{\xi}}

\newcommand{\comp}{\#_{\textsc{kn}}}

\newcommand\Ident{\operatorname{id}}
\newcommand\Os{\operatorname{Os}-}

\renewcommand\Re{\operatorname{Re}\,}
\newcommand\expectation{{\bf E}}
\newcommand\probability{{\bf P}}
\newcommand\fourier{{\bf F}}

\newcommand{\intrd}{\int_{\R^d}}
\newcommand{\jb}[2][\psi]{\left\langle #2 \right\rangle_{#1}}
\newcommand{\lv}{\lesssim}
\newcommand{\gv}{\gtrsim}

\title[Constructing Markov processes using Feynman path integrals]{A pseudo-differential operator construction of Markov processes using Feynman path integrals}
\author{Alexander Potrykus}

\address{Department of Mathematics\\
Swansea University\\
Singleton Park\\
Swansea SA2 8PP\\
UK
}
\email{
 \href{mailto:A.K.K.Potrykus@swansea.ac.uk}{A.K.K.Potrykus@swansea.ac.uk}
}

\urladdr{\href{http://www-maths.swan.ac.uk/staff/akp}{http://www-maths.swan.ac.uk/staff/akp}
}



\begin{document}
\maketitle
\frenchspacing
\begin{abstract}
In this paper pseudo-differential operators with negative definite symbols are used to construct time- and space-inhomogeneous Markov processes. This is achieved by using the Markov evolution system associated with the fundamental solution of the corresponding pseudo-differential evolution equation. Negative definite symbols are non-standard and differ significantly from the class of H\"{o}rmander type symbols. The novelty of this work is the derivation and the representation of the fundamental solution as a Feynman path integral. This implies that the transition function of the constructed Markov process can be written as a pseudo-differential operator that has a Feynman path integral as its symbol. 
\end{abstract}

\section{Introduction}\label{secIntroduction}
We consider the parabolic pseudo-differential evolution equation
\begin{align}\label{eqPDE}
& \partial_t u(t;x)+a(t;x,D_x)u(t;x)=f(t;x)\qquad\text{for } 0\leq s<t\leq T\\
& u(s,x)=u_0(x),
\end{align}
where $u$ is a real valued function with $x\in \R^d$ and $f$ and $u_0$ belong to suitable function spaces. The operator $a(t;x,D_x)$ is pseudo-differential operator with negative definite symbol $a(t;x,\xi)$, defined on $\mathcal{S}(\R^d)$ by
\begin{align}\label{eqPSDO}
a(t;x,D_x)u(x)=\int_{\R^d}e^{ix\cdot \xi}a(t;x,\xi)\wh{u}(\xi)\text{\dj}\xi,
\end{align}
and $\text{\dj}\xi:=(2\pi)^{-\frac{d}{2}}\id\xi$. Recall that negative definite functions have a L\'evy-Khintchine representation, cf. \eqref{eqLevyKhintchine}. The purpose of this paper is to use the fundamental solution $e(t,s,x,D_x)$ of \eqref{eqPDE} to construct a time- and space-inhomogeneous Markov process $(X_t)_{t\geq 0}$ and to derive a formula for the transition function $p_{s,t}(x,A)$, $0\leq s\leq t<\infty$, $x\in \R^d$, $A\in\mathcal{B}(\R^d)$ in terms of the symbol $a(t;x,\xi)$. The symbol of the transition function can be written in terms of a Feynman path integral. For more information regarding mathematical approaches to Feynman path integrals cf. \cite{Albeverio-08,Johnson-00,Mazzucchi-09,Elworthy-84,Kolokoltsov-99,Fujiwara-80}. Note that negative definite functions do not fit into any of the classical symbol classes. Therefore standard theory cannot be applied.

It is possible to associate with a Markov process $(X_t)_{t\geq 0}$  a family of operators $(T_{s,t})_{,0\leq s\leq t<\infty}$ by setting
\begin{align}\label{eqEvolutionFamilyMarkov}
T_{s,t}f(x)=\expectation[f(X_t)|X_s=x]
\end{align}
for each $x\in \R^d$ and $f\in B_b(\R^d)$, the space of all bounded Borel measurable functions. 
\begin{definition}\label{definitionMarkovEvolutionFamily}
A family of operators $(T_{s,t})_{0\leq s\leq t<\infty}$ is called a Markov evolution family if it has the properties:
\begin{enumerate}
\item[a)] $T_{s,t}$ is a linear operator on $B_b(\R^d)$ for $0\leq s\leq t<\infty$,
\item[b)] $T_{s,s}=\Ident$ for $s\geq 0$,
\item[c)] $T_{r,s}\circ T_{s,t}=T_{r,t}$ for $0\leq r\leq s\leq t<\infty$,
\item[d)] $f\geq 0 \Rightarrow T_{s,t}f\geq 0$ for all $0\leq s\leq t<\infty$, $f\in B_b(\R^d)$,
\item[e)] $T_{s,t}$ is a contraction, i.e. $\|T_{s,t}\|\leq 1$ for $0\leq s\leq t<\infty$,
\item[f)] $T_{s,t}(1)=1$ for all $t\geq 0$.
\end{enumerate}
\end{definition}
It is easy to see that the family of operators given by \eqref{eqEvolutionFamilyMarkov} is a Markov evolution family in the above sense.
If $p_{s,t}(x,A)$ is the transition function of the process $(X_t)_{t\geq 0}$, then
\begin{align}
T_{s,t}f(x)=\int_{\R^d}f(y)p_{s,t}(x,\id y)
\end{align}
for $f\in B_b(\R^d)$ and $x\in \R^d$. Therefore
\begin{align}\label{eqTransitionFunction}
p_{s,t}(x,A)=T_{s,t}\chi_A(x)=\probability\big(X(t)\in A|X(s)=x\big).
\end{align}
If the Markov process has a transition density $p_{s,t}(x,y)$ then
\begin{align}
p_{s,t}(x,A)=\int_A p_{s,t}(x,y)\id y.
\end{align}
The reverse problem is now of interest: start with an operator $a(t;x,D_x)$ and show that the fundamental solution to \eqref{eqPDE} defines a Markov evolution family. In order to do this, it is necessary that the operator $a(t;x,D_x)$ has negative definite symbols. A family of operators $e(t,s;x,D_x)$ is called a fundamental solution of \eqref{eqPDE} if
\begin{align}
\begin{cases}
&\partial_t e(t,s;x,D_x)+a(t;x,D_x)e(t,s;x,D_x)=0,\qquad\text{on } (s,T],\\
&e(s,s,;x,D_x)=\Ident.
\end{cases}
\end{align}
As a first idea to find the fundamental solution, consider on $\mathcal{S}(\R^d)$ the operator 
\begin{align}
e^{-\int_s^t a(\tau)\id \tau}(x,D_x)u(x)=\int_{\R^d} e^{ix\cdot \xi} e^{-\int_s^t a(\tau;x,\xi)\id \tau}\wh{u}(\xi) \text{\dj}\xi.
\end{align}
Clearly, this operator is not the fundamental solution of \eqref{eqPDE}. However we now use Feynman's time slicing approximation \cite{Feynman-48}: define for $k\in\N$ a partition
\begin{align}\label{eqPartition}
\pi_{s,t}=\{t_0=s,t_1,t_2,\ldots,t_{k+1}=t\}
\end{align}
where $0\leq t_0\leq t_1\leq \ldots\leq t_{k+1}\leq T$. Denote by $|\pi_{s,t}|=\max_{1\leq j\leq k+1}|t_j-t_{j-1}|$ the mesh of $\pi_{s,t}$. Consider now $e^{-\int_{t_j}^{t_{j+1}}a(\tau)\id\tau}(x,D_x)$ as the fundamental solution to operators $\partial_t+a(t_{j+1};x,D_x)$ with ``frozen'' time-dependence. We want to show that
\begin{align}\label{eqIntroduction1}
e^{-\int_{t_0}^{t_1}a(\tau)\id \tau}(x,D_x)\circ \ldots \circ e^{-\int_{t_k}^{t_{k+1}}a(\tau)\id \tau}(x,D_x) 
\end{align}
converges for $|\pi_{s,t}|\rightarrow 0$ to the fundamental solution $e(t,s;x,D_x)$ of \eqref{eqPDE}. Using oscillatory integrals  \eqref{eqIntroduction1} can be written as
\begin{align}\label{eqIntroduction2}
\Os\int_{\R^{2(k+1)d}}e^{\sum_{j=0}^k i(x^{j+1}-x^{j})\cdot \xi^{j+1}-\int_{t_j}^{t_{j+1}}a(\tau;x^j,\xi^{j+1})\id \tau}
u(x^{k+1})\text{\dj} x^{k+1}\text{\dj} \xi^{k+1}\cdots \text{\dj} x^1\text{\dj} \xi^1.
\end{align}
For $|\pi_{s,t}|\rightarrow 0$, the number of integrals in \eqref{eqIntroduction2} will to go to infinity. 
With a version of Kumano-go's theory of multiple symbols \cite{Kumanogo-75} it is possible to estimate multiple compositions of pseudo-differential operators such as \eqref{eqIntroduction2}. This theory is extended in Section \ref{secMultipleSymbols} to negative definite symbols.



In order to see how our symbol classes differ from the standard H{\"o}rmander symbol classes and their adaption to basic weight functions, let us give some details concerning negative definite functions, cf. \cite{Jacob-01:2}.

 The following notation will be used throughout: for $\alpha\in\N_0^d=\N^d\cup \{0\}$ the derivative $\partial_x^{\alpha}u$ is defined as $\frac{\partial^{|\alpha|}}{\partial x_1^{\alpha_1}\cdots \partial x_d^{\alpha^d}}u$. For two quantities $X$ and $Y$ we use $X\lv Y$ or $Y\gv X$ to denote the statement $X\leq CY$ or $X\geq CY$. We use subscripts to emphasize the dependence on parameters, i.e. $X\lv_k Y$ is synonymous with $X\leq C_k Y$ for some constant $C_k>0$ that depends on the parameter $k$. Furthermore, we use the Japanese bracket convention $\jb{\xi}:=\big(1+\psi(\xi)\big)^{\frac{1}{2}}$. The corresponding operator will similarly be written as $\jb{\mathrm{D}}$. Also, define for $g=0,1,2$, the cut-off function $\rho_g:\N_0\rightarrow \N_0,\:\rho_g(k)=\min(k,g)$. 

A function $\psi:\R^d\rightarrow \C$ is a continuous negative definite function if it has the L\'evy-Khintchine representation
\begin{align}\label{eqLevyKhintchine}
\psi(\xi)=c+il\cdot \xi+\frac{1}{2}\xi\cdot Q\xi+\int_{\R^d\backslash\{0\}} \left(1-e^{i\xi\cdot y}+\frac{i\xi\cdot y}{1+|y|^2}\right)\nu(\id y),
\end{align} 
where $c\geq 0$, $l\in\R^d$, $Q\in\R^{d\times d}$ is a positive semi-definite matrix, and $\nu$ is a L\'evy measure, i.e.
\begin{align}
\int_{\R^d\backslash \{0\}} \min(|y|^2,1)\nu(\id y)<\infty.
\end{align}
In particular this implies that $\psi$ is in general not smooth and it is not possible to define a principal symbol. Also, no homogeneous expansion formulae exist. Every continuous negative definite function $\psi$ is the characteristic exponent of a L\'evy process $(Y_t)_{t\geq 0}$:
\begin{align}
\expectation[e^{i\xi\cdot Y_t}]=e^{-t\psi(\xi)}.
\end{align}
The pseudo-differential operator $-\psi(D)u=-\fourier^{-1}[\psi\wh{u}]$ corresponds on $C_0^{\infty}(\R^d)$ with the infinitesimal generator of the L\'evy process $(Y_t)_{t\geq 0}$. The associated operator semigroup $(T_t)_{t\geq 0}$ has the representation $T_tu(x)=\fourier^{-1}[e^{-t\psi}\wh{u}]$. 
Furthermore, pseudo-differential operators with negative definite symbols satisfy the positive maximum principle, cf. \cite{Courrege-67}. This motivates the use of continuous negative definite functions in the construction of more general time- and space-inhomogeneous Markov processes.
\begin{definition}\label{defSymbolClass}
Let $\psi:\R^d\rightarrow \R$ be a continuous negative definite function satisfying for all multi-indices $\alpha\in \N_0^d$,
\begin{align}
|\partial_{\xi}^{\alpha}\jb{\xi}^2|\lv_{\alpha} \jb{\xi}^{2-\rho_2(|\alpha|)}.
\end{align}
For $m\in \R$ we call a $C^{\infty}$-function $a:\R^{2d}\rightarrow \C$ a symbol of class $S^{m,\psi}_{\rho_g}(\R^{2d})$ , $g\in \{0,1,2\}$, if for all $\alpha,\beta\in\N_0^d$,
\begin{align}\label{eqIntroduction3}
|\partial_{\xi}^{\alpha}D_x^{\beta}a(x,\xi)|\lv_{\alpha,\beta} \jb{\xi}^{m-\rho_g(|\alpha|)}
\end{align}
where $x,\xi\in\R^d$. 
\end{definition} 
Note that the use of the cut-off function $\rho_g$ on the right-hand side of \eqref{eqIntroduction3} means that
the decay of the derivatives of the symbol $a$ only improves upto derivatives of order $2$: for a symbol $a\in S^{m,\psi}_{\rho_2}(\R^{2d})$ and $\alpha,\alpha',\alpha''\in \N_0^d$, $|\alpha|=1$, $|\alpha'|=2$, $|\alpha''|>2$ it follows that $\partial_{\xi}^{\alpha}a \in S^{m-1,\psi}_{\rho_1}(\R^{2d})$, $\partial_{\xi}^{\alpha'}a\in S^{m-2,\psi}_{\rho_0}(\R^{2d})$ and $\partial_{\xi}^{\alpha''}a \in S^{m-2,\psi}_{\rho_0}(\R^{2d})$.  This is why the symbolic calculus based on basic weight functions \cite{Kumanogo-63} cannot be applied.

The construction of a fundamental solution of \eqref{eqPDE} requires the following conditions on the symbol $a(t;x,\xi)$.
\begin{assumptions}\label{mainAssumptions}
For $T>0$, let $a:[0,T]\times \R^{2d}\rightarrow \C$ satisfy the following conditions:
\begin{itemize}
\item[(A1)] For $m\leq 2$,
\begin{align}
 a\in C_b\big([0,T]; S^{m,\psi}_{\rho_2}(\R^{2d})\big).
\end{align}
\item[(A2)] There exists $0\leq m'\leq 2$ such that on $[0,T]\times \R^{2d}$
\begin{equation}
\Re a(t;x,\xi)\gv  \jb{\xi}^{m'}
\end{equation}
uniformly in $t\in [0,T]$.
\item[(A3)] For any multi-indices $\alpha,\beta\in \N_0^d$, we have on $[0,T]\in \R^{2d}$,
\begin{equation}\label{A3}
\left|\frac{\partial_{\xi}^{\alpha}D_x^{\beta}a(t;x,\xi)}
{\Re a(t;x,\xi)}\right|\lv_{\alpha,\beta} \jb{\xi}^{-\rho_2(|\alpha|)}
\end{equation}
uniformly in $t\in [0,T]$.
\end{itemize}
\end{assumptions}
If $m'=m$, then (A3) follows from (A2). Note that (A2) implies that parabolic operators of degenerate type satisfy these assumptions. The restriction to $m\leq 2$ in (A1) follows from the fact that every continuous negative 
definite function satisfies the estimate $|\psi(\xi)|\lv_{\psi}\jb[|\cdot|^2]{\xi}^2$, cf. \cite{Hoh-98}. Continuous negative definite functions are needed for the construction of Markov processes, cf. Theorem \ref{theoremTransitionFunction}. This is the motivation for using the symbol class $S^{m,\psi}_{\rho_2}(\R^{2d})$ in (A1). Let us state the main results of this paper. The first one concerns the existence and representation of the fundamental solution of \eqref{eqPDE}.
\begin{theorem}\label{theoremFundamentalSolution}
For $T>0$, let $a:[0,T]\times \R^{2d}\rightarrow \C$ be a function satisfying Assumptions \ref{mainAssumptions}.
Then there exists a symbol $p^{\lim}(s,t)\in S^{0,\psi}_{\rho_2}(\R^{2d})$, given by
\begin{align}\label{eqPLim}
p^{\lim}(s,t;x,\xi)=\lim_{|\pi_{s,t}|\rightarrow 0} \Os\int_{\R^{2kd}} e^{-i\sum_{j=1}^k\big(y^j\cdot \eta^j-\int_{t_j}^{t_{j+1}}a(\tau;x+\sum_{l=1}^j y^l,\xi+\eta^{j+1})\id \tau\big)}\text{\dj} y^k\text{\dj} \eta^k\ldots \text{\dj} y^1\text{\dj} \eta^1.
\end{align}
The corresponding pseudo-differential operator $p^{\lim}(s,t;x,D_x)$ is the fundamental solution for \eqref{eqPDE} and can be written for $u\in \mathcal{S}(\R^d)$ as
\begin{align}\label{eqPLimOperator}
\begin{split}
&p^{\lim}(s,t;x,D_x)u(x)\\
=\:&\lim_{|\pi_{s,t}|\rightarrow 0} e^{-\int_s^{t_1}a(\tau)\id \tau}(x,D_x)e^{-\int_{t_1}^{t_2}a(\tau)\id \tau}(x,D_x)\circ \ldots\circ e^{-\int_{t_k}^t a(\tau)\id \tau}(x,D_x)u(x).
\end{split}
\end{align} 
\end{theorem}
Using a standard argument, it is possible to show that Theorem \ref{theoremFundamentalSolution} holds even if the symbol $a$ is non-smooth.
\begin{corollary}\label{corollaryRoughSymbols}
Let $a:[0,T]\times \R^{2d}\rightarrow \C$ be the symbol from Theorem \ref{theoremFundamentalSolution}. Introduce the non-smooth symbol class $S^{m,\psi}_{\rho_2}(\lambda,\R^{2d})$, $\lambda \in\N_0\cup \{\infty\}$ consisting of all functions
$b:\R^{2d}\rightarrow \C$ such for all $|\alpha|,|\beta|\leq \lambda$,
\begin{align}
|\partial_{\xi}^{\alpha}D_x^{\beta} b(x,\xi)|\lv_{\alpha,\beta} \jb{\xi}^{m-\rho_2(|\alpha|)}.
\end{align}
Replace (A2) from Assumptions \ref{mainAssumptions} by
\begin{itemize}
\item[(A2')] For $m\leq 2$ we have
\begin{align}
a\in C_b\big([0,T];S^{m,\psi}_{\rho_2}(100d,\R^{2d})\big).
\end{align}
\end{itemize}
Then there exists a symbol $p^{\lim}(s,t)\in S^{0,\psi}_{\rho_2}(100d,\R^{2d})$ given by \eqref{eqPLim} such that the corresponding pseudo-differential operator \eqref{eqPLimOperator} is the fundamental solution of \eqref{eqPDE}.
\end{corollary}
The factor $100d$ is an upper estimate for the differentiability of the symbol and emphasizes the fact that it depends on the dimension $d$ of the underlying space. With the techniques used in this paper it is
not possible to achieve a bound that is independent of $d$. 

The next theorem gives the representation of the transition function of a Markov process in terms of the symbol $p^{\lim}(s,t;x,\xi)$.
\begin{theorem}\label{theoremTransitionFunction}
For $T>0$, let $a:[0,T]\times \R^{2d}\rightarrow \C$ be a function satisfying:
\begin{itemize}
\item[(B1)] For $t\in [0,T]$ and $x\in\R^d$, $\xi\mapsto a(t;x,\xi)$ is a continuous negative definite function such that $a(t;x,0)=0$.
\item[(B2)] For $m\leq 2$,
\begin{align}
a\in C_b\big([0,T]; S^{m,\psi}_{\rho_2}(\R^{2d})\big).
\end{align}
\item[(B3)] There exists $0\leq m'\leq 2$ and $R>0$ such that for $x,\xi\in \R^d$, $|\xi|\geq R$,
\begin{equation}
\Re a(t;x,\xi)\gv  \jb{\xi}^{m'}
\end{equation}
uniformly in $t\in [0,T]$.
\item[(B4)] There exists $R>0$ such that for any multi-indices $\alpha,\beta\in \N_0^d$ and $x,\xi\in \R^d$, $|\xi|\geq R$,
\begin{equation}\label{B4}
\left|\frac{\partial_{\xi}^{\alpha}D_x^{\beta}a(t;x,\xi)}
{\Re a(t;x,\xi)}\right|\lv_{\alpha,\beta} \jb{\xi}^{-\rho_2(|\alpha|)}
\end{equation}
uniformly in $t\in [0,T]$.
\end{itemize}
Then there exists a Markov process with transition function
\begin{align}\label{eqTransitionFunction}
p_{s,t}(x,A)=\int_{\R^d}e^{ix\cdot \xi}p^{\lim}(s,t;x,\xi)\wh{\chi_A}(\xi)\text{\dj}\xi
\end{align}
where $\wh{\chi_A}$ has to be understood as an approximation and $p^{\lim}$ is given by \eqref{eqPLim}.
\end{theorem}
Note that the transition function is a Feynman path integral as  $p^{\lim}(s,t;x,\xi)$ is an infinite-dimensional integral. Furthermore, it follows from \eqref{eqPLimOperator} that
\begin{align}
p^{\lim}(s,t;x,\xi)= e^{-\int_s^t a(\tau;x,\xi)\id \tau}+\lim_{|\pi_{s,t}|\rightarrow 0}r_1(\pi_{s,t};x,\xi).
\end{align}
where $r_1(s,t)\in S^{-1,\psi}_{\rho_1}$. If $a$ is a time-independent symbol and an equidistant partition $\pi_{s,t}$ is chosen then 
\begin{align}
p^{\lim}(s,t;x,D_x)=\lim_{n\rightarrow \infty} \big(e^{-\frac{a}{n}}(x,D_x)\big)^n
\end{align}
which can be compared to the Trotter-Kato-Chernoff product formula for strongly continuous semigroups, cf. \cite{Chernoff-74,Kato-66}. This formula has been used in \cite{Butko-10} for an approximation result for strongly continuous contraction semigroups that are positivity preserving on $C_{\infty}(\R^d)$ and that have as pregenerators pseudo-differential operators with negative definite symbols. 

Our approach is mainly based on the papers by N Kumano-go and Fujiwara \cite{KumanogoNaoto-96,Fujiwara-79,Fujiwara-80,Fujiwara-91,Fujiwara-93}. In \cite{KumanogoNaoto-96} a similar result to that of Theorem \ref{theoremFundamentalSolution} is given for symbols of the class $S^m_{\lambda,\rho,\delta}(\R^{2d})$. This class is an extension of the H{\"o}rmander symbol class to basic weight functions as introduced in \cite{Kumanogo-63}. 
As will be seen later, a continuous negative definite function $\psi$ is not a basic weight function. Therefore $S^{m,\psi}_{\rho_g}(\R^{2d})$ is not a subset of $S^{m}_{\lambda,\rho,\delta}(\R^{2d})$ and techniques need to be suitably modified. Let us emphasize that the reason for using negative definite functions is the connection to the theory of Markov processes. The symbol class $S^{m,\psi}_{\rho_g}(\R^{2d})$ was introduced in \cite{Hoh-98}, cf. also \cite{Hoh-00},\cite{Jacob-02}, to construct Feller semigroups and Feller processes. In \cite{Potrykus-09},\cite{Potrykus-10:1}, it was adapted to take into account rough, non-smooth, symbols similar to the ones found in Corollary \ref{corollaryRoughSymbols}. The important ideas used in  \cite{KumanogoNaoto-96} which have to be modified in order to work with the symbol class $S^{m,\psi}_{\rho_g}(\R^{2d})$ are:
\begin{itemize}
\item The theory of multiple symbols as developed in \cite{Kumanogo-75} gives estimates for $k$-fold compositions of pseudo-differential operators in terms of their symbols. In particular, the dependence of any constants on the variable $k$ is stated explicitly. It is this theory that is extended in Section \ref{secMultipleSymbols} to the symbol classes $S^{m,\psi}_{\rho_g}(\R^{2d})$. Note that as a by-product of this theory a precise control on the maximum differentiability of the involved symbols is possible. More details are given in Section \ref{secMultipleSymbols}.
\item The calculation of the remainder term of the symbol of a $k$-fold composition of pseudo-differential operators as introduced in \cite{Fujiwara-91} by Fujiwara is crucial for our argumentation. This leads to sequences that skip every other index and related estimates. Details are provided in Appendix \ref{secFujiwara} and Appendix \ref{secSkip}.  
\end{itemize}
Many results concerning the construction of Markov processes using continuous negative definite functions exist, c.f. \cite{Jacob-01:2,Jacob-02,Jacob-05:1}. In \cite{Maslov-73} it is shown that for certain negative definite symbols $a(x,\xi)$ with associated Feller semigroup $(T_t)_{t\geq 0}$, the operators $T_t$ are on $C_0^{\infty}(\R^d)$ pseudo-differential operators with symbol $p(t;x,\xi)$ satisfying for $t\rightarrow 0$ the asymptotic relation
$$
p(t;x,\xi)=e^{-ta(x,\xi)}+o(1)
$$
uniformly for $x\in K$, $K\subset \R^d$ compact, and $\xi\in\R^d$.  In \cite{Kolokoltsov-00}, Kolokoltsov constructed Markov processes by using symbols $a(x,\xi)$ given by
\begin{align*}
a(x,\xi)=i\big(b(x),\xi\big)+\int_0^{\infty}\int_{S^{d-1}} \left(e^{-iy\cdot \xi}
-1+\frac{iy\cdot \xi}{1+|y|^2}\right)\frac{\mathrm{d}|y|}{|y|^{1+\alpha}}\wt{\mu}(x,\id \eta)
\end{align*}
where $\alpha\in (0,2)$, $y=|y|\eta$ and $\wt{\mu}(x,\id \eta)$ is a kernel on $\R^d\times \mathcal{B}(S^{d-1})$ that satisfies some additional assumptions. In \cite{Jacob-07} an approximation for a Feller semigroup $(T_t)_{t\geq 0}$ based on the Yosida approximation is given:  let $a(x,D)$ be a pseudo-differential operator with continuous negative definite symbol $a(x,\xi)$ satisfying the usual assumptions. Define the Yosida approximation of the symbol as
$
a^{\nu}(x,\xi)=(\nu a(x,\xi)(\nu+a(x,\xi))^{-1}.
$
This symbol is uniformly bounded in $(x,\xi)\in \R^{2d}$, and hence the associated semigroups $(T_t^{\nu})_{t\geq 0}$ exist. If $-a(x,D)$ is the pre-generator of $(T_t)_{t\geq 0}$, then $T_tu=\lim_{\nu\rightarrow \infty}T_t^{\nu}u$ strongly for $t>0$.  In \cite{Kochubei-88} a Markov process is constructed using the fundamental solution for \eqref{eqPDE}. Here, the symbol $a(t;x,\xi)$ has the representation
\begin{align}
a(t;x,\xi)=\sum_{j=0}^m a_j(t;x,\xi)
\end{align}
where each $a_j$ has certain homogeneity as well as other properties
that guarantee that the theory of hypersingular integral operators is applicable. Note that these conditions
imply that the involved symbols are continuous negative definite functions.
In \cite{Tsutsumi-74} Tsutsumi used the Levi-Mizohata method to find the fundamental solution for \eqref{eqPDE} for \psd operators with H{\"o}rmander symbols $S^{m}_{\rho,\delta}(\R^{2d})$ for $m\geq 0$ and $0\leq \delta\leq \rho\leq 1$. In short this method can be described as follows: start by setting
\begin{align}
e_0(t,s;x,\xi):=e^{-\int_s^t a(\tau;x,\xi)\id \tau}
\end{align}
and define for $j=1,2,\ldots$ the symbols $e_j(t;s,x,\xi)$ as solutions to the ordinary differential equation
\begin{align}
\begin{cases}
&\partial_t e_j(t,s;x,\xi)+a(t;x,\xi)e_j(t,s;x,\xi)=-q_j(t,s;x,\xi)\\
&e_j(t,s;x,\xi)|_{t=s}=0,
\end{cases}
\end{align}
where
\begin{align}
q_j(t,s;x,\xi):=\sum_{k=0}^{j-1} \sum_{|\alpha|+k=j} \frac{1}{\alpha!}
\partial_{\xi}^{\alpha} a(t;x,\xi) D_x^{\alpha}e_k(t,s;x,\xi).
\end{align}
Using an iterative procedure, the symbol of the fundamental solution $e(t,s;x,D_x)$ for \eqref{eqPDE} can then be written as
\begin{align}\label{eqIntroduction4}
e(t,s;x,\xi)=e^{-\int_s^t a(\tau;x,\xi)\id \tau}+r(t,s;x,\xi)
\end{align}
where $r(t,s;x,\xi)$ is of lower order. This result was extended in \cite{Bottcher-08} to the symbol classes $S^{m,\psi}_{\rho_g}(\R^{2d})$, $0\leq m\leq 2$, using slightly stronger conditions on the symbols than we use in Assumptions \ref{mainAssumptions}. It is shown that the fundamental solution forms a Markov evolution system and therefore gives rise to a time- and space-inhomogeneous Markov process. In this case, the corresponding Markov transition function is a pseudo-differential operator with symbol of class $S^{0,\psi}_{\rho_g}(\R^{2d})$ and representation \eqref{eqIntroduction4}. However, the exact form of the remainder term $r(t,s;x,\xi)$ is not known. In this paper we weaken the conditions from \cite{Bottcher-08} to allow non-degenerate, non-smooth, symbols, and show that the fundamental solution can be written as a Feynman path integral.

This paper is organized as follows: Section \ref{secMultipleSymbols} develops the theory of multiple symbols from \cite{Kumanogo-75} for negative definite symbols $S^{m,\psi}_{\rho_g}(\R^{2d})$. The construction of the fundamental solution of \eqref{eqPDE} is given in Section \ref{secFundamentalSolution} and follows \cite{KumanogoNaoto-96}. In Section \ref{secMarkovProcesses}, the fundamental solution constructed in Section \ref{secFundamentalSolution} is used to construct Markov processes. Finally, the appendix contains results that do not naturally fit into any of the other proofs but hopefully provide additional insights.

\section{Pseudo-differential operators with multiple negative definite symbols}\label{secMultipleSymbols}
In \cite{Kumanogo-75} a theory for the symbol of multiple compositions of pseudo-differential operators with symbols of class $S^{m}_{\lambda,\rho,\delta}(\R^{2d})$ is developed. As mentioned in Section \ref{secIntroduction}, the class $S^m_{\lambda,\rho,\delta}(\R^{2d})$ is an extension of the H{\"o}rmander symbol classes to functions $\lambda:\R^d\rightarrow \R$ with the properties:
\begin{enumerate}
    \item[(a)]$1\lv \lambda(\xi)\lv \jb[|\cdot|^2]{\xi}$
    \item[(b)]$|\partial_{\xi}^{\alpha}\lambda(\xi)|\lv_{\alpha} \lambda(\xi)^{1-|\alpha|},\qquad \alpha\in\N_0^d$.
\end{enumerate}
Continuous negative definite functions are not weight functions in the above sense, i.e. we cannot set $\lambda(\xi):=\jb{\xi}$ as by Definition \ref{defSymbolClass}, property (b) is not satisfied. For this reason Hoh extended in \cite{Hoh-98} the symbolic calculus for weight functions to take into account negative definite functions. 

Consider the composition of two pseudo-differential operators with symbols $a_1\in S^{m_1,\psi}_{\rho_g}(\R^{2d})$ and $a_2\in S^{m_2,\psi}_{\rho_g}(\R^{2d})$, $g\in \{0,1,2\}$, in Kohn-Nirenberg quantization. The symbol of the composition can be written as an oscillatory integral:
\begin{align}
(a_1\#_{\textsc{kn}} a_2)(x,\xi)=\Os\int_{\R^{2d}} e^{-iy\cdot \eta}a_1(x,\xi+\eta)a_2(x+y,\xi)\text{\dj} y\text{\dj} \eta
\end{align}
In a straightforward manner it is possible to find a similar expression for the symbol of the composition of $k\in \N$ pseudo-differential operators. The following notational conventions are helpful: let $x^1, \ldots, x^k\in \R^d$ and  $\xi^1,\ldots, x^k\in \R^d$, then
\begin{align}
\wt{x}_k=(x^1,\ldots,x^k)\in \R^{kd}, \: \wt{\xi}_k=(\xi^1,\ldots,\xi^k)\in \R^{kd}.
\end{align}
Also,
\begin{align}
\wt{x}_k\cdot \wt{\xi}_k=\sum_{j=1}^k x^j\cdot \xi^j
\end{align}
and for $1\leq l\leq k$,
\begin{align}
\summ{x}{k}{l}=\sum_{j=1}^l x^j.
\end{align}
Moreover,
\begin{align}
\id \wt{x}_k\id \wt{\xi}_k=\id x^1\id \xi^1\cdots \id x^k\id \xi^k
\end{align}
and for multi-indices $\wt{\alpha}_k=(\alpha^1,\ldots,\alpha^k)\in \N_0^{kd}$, $\wt{\beta}_k=(\beta^1,\ldots,\beta^k)\in \N_0^{kd}$,
\begin{align}
\partial_{\wt{\xi}_k}^{\wt{\alpha}_k}D_{\wt{x}_k}^{\wt{\beta}_k}=\partial_{\xi^1}^{\alpha^1}\cdots \partial_{\xi^k}^{\alpha^k}D_{x^1}^{\beta^1}\cdots D_{x^k}^{\beta^k}
\end{align}
where $D^{\alpha}=(-i)^{|\alpha|} \partial^{\alpha}$. The symbol of the composition of $k$ pseudo-differential operators with symbols $a_j\in S^{m_j,\psi}_{\rho_g}(\R^{2d})$ is then given by
\begin{align}
&(a_1\#_{\textsc{kn}} \ldots \#_{\textsc{kn}} a_k)(x,\xi)\\
=&\Os\int_{R^{2(k-1)d}} e^{-i\wt{y}_{k-1}\cdot \wt{\eta}_{k-1}}a_1(x,\xi+\eta^1)
a_2(x+\summ{y}{k-1}{1},\xi+\eta^2)\cdots a_k(x+\summ{y}{k-1}{k-1},\xi)\text{\dj}\wt{y}_{k-1}\text{\dj}\wt{\eta}_{k-1}.
\end{align}
The theory of multiple symbols therefore includes for $k=2$ the theory of double symbols as given in \cite{Hoh-98} for negative definite functions. The remainder of this section is concerned with proving estimates for multiple negative definite symbols and follows \cite{Kumanogo-75}. Often the following extension of Peetre's inequality to negative definite functions is needed, cf. \cite{Jacob-02}.
\begin{lemma}\label{lemmaPeetre}
Let $\psi:\R^d\rightarrow \C$ be a negative definite function. Then 
\begin{align}
\frac{1+|\psi(\xi)|}{1+|\psi(\eta)|}\leq 2 (1+|\psi(\xi-\eta)|).
\end{align}
\end{lemma}
The following estimate is important for the existence of oscillatory integrals based on negative definite functions, cf. \cite{Jacob-02}.
\begin{lemma}\label{lemmaGrowth}
Any locally bounded negative definite function $\psi:\R^d\rightarrow \C$ satisfies the estimate
\begin{align}
|\psi(\xi)|\lv_{\psi} \jb[|\cdot|^2]{\xi}^2.
\end{align}
\end{lemma}
The next definition describes the class of multiple negative definite symbols.
\begin{definition}
Let a continuous negative definite function $\psi:\R^d\rightarrow \R$ satisfy for all multi-indices $\alpha\in \N_0^d$,
\begin{align}
|\partial_{\xi}^{\alpha}\jb{\xi}^2|\lv_{\alpha} \jb{\xi}^{2-\rho_2(|\alpha|)}.
\end{align}
Then for $\wt{m}_k=(m_1,\ldots,m_k)\in \R^k$, $k\in \N$, we say that a $C^{\infty}$-function
\begin{align}
a:\R^{kd}\times \R^{kd}\rightarrow \C, (\wt{x}_k,\wt{\xi}_k)\mapsto p(\wt{x}_k,\wt{\xi}_k):=p(x^1,\xi^1,\ldots,x^k,\xi^k)
\end{align}
belongs to the class of multiple symbols $S^{\wt{m}_k,\psi}_{\rho_g}(\R^{2kd})$, $g\in \{0,1,2\}$, if for any multi-indices, $\wt{\alpha}_k$, $\wt{\beta}_k\in \N_0^{kd}$, we have
\begin{align}
\left|\partial_{\bxi_k}^{\ba_k} D_{\bx_k}^{\bb_k} a(\bx_k,\bxi_k)\right|
&\lv_{\ba_k,\bb_k} \prod_{j=1}^k \jb{\xi^j}^{m_j-\rho_g(|\alpha^j|)}. 
\end{align}
\end{definition}
Using the semi-norms
\begin{align}
|a|_{l,l'}^{(\bm_k)}:=\max_{|\ba_{k}|\leq l,|\bb_{k}|\leq l'}\sup_{\bx_{k},\bxi_{k} \in \R^{kd}}\left(
\left|\partial_{\bxi_{k}}^{\ba_{k}} 
D_{\bx_{k}}^{\bb_{k}} a(\bx_k,\bxi_{k})\right|\prod_{j=1}^k
\jb{\xi^j}^{-m_j+\rho_g(|\alpha^j|)}\right),
\end{align}
where $l,l'\in \N_0$, the space $S^{\bm_{k},\psi}_{\rho_g}(\R^{2kd})$ is a Fr\'echet-space. The following Lemma is needed for the proof of Theorem \ref{theoremThetaEstimate}. 
\begin{lemma}\label{lemmaAux1}
Let $\psi:\R^d\rightarrow \R$ be a continuous negative definite function such that for all multi-indices $\alpha\in \N_0^d$,
\begin{align}\label{eqSection2_1}
|\partial_{\xi}^{\alpha}\jb{\xi}^2|\lv_{\alpha} \jb{\xi}^{2-\rho_2(|\alpha|)}.
\end{align}
Then there exists a constant $c_0>0$ such that
\begin{align}
\frac{1}{2}\jb{\xi}\leq \jb{\xi+\eta}\leq 2\jb{\xi}
\end{align}
for $|\eta|\leq c_0\jb{\xi}$.
\end{lemma}
\begin{proof}
We find $\jb{\xi+\eta}-\jb{\xi}=\sum_{j=1}^d \int_0^1 \eta_j \jb[\partial_{\xi_j} \psi]{\xi+\theta \eta}\id \theta$. Using \eqref{eqSection2_1} and Peetre's inequality for negative definite functions, cf. Lemma \ref{lemmaPeetre}, this can now be estimated as
\begin{align}
|\jb{\xi+\eta}-\jb{\xi}|\lv \sum_{j=1}^d \int_0^1|\eta_j \jb{\xi+\theta\eta}^0|\id \theta\leq c_1 |\eta|.
\end{align}
Setting $c_0:=(c_1)^{-1}$ concludes the proof.
\end{proof}

\begin{theorem}\label{theoremThetaEstimate}
Let $a\in S^{\wt{m}_k,\psi}_{\rho_g}(\R^{2kd})$, $g\in \{0,1,2\}$, be a multiple negative definite symbol and define the symbol $b_{\theta}$, $|\theta|\leq 1$, by
\begin{align}
&b_{\theta}(x,\xi)=\int_{\R^{2(k-1)d}}e^{-i \by_{k-1}\cdot \be_{k-1}}
 a(x,\xi+\theta \eta^1,x+\summ{y}{k-1}{1},\xi+\theta\eta^2,\ldots, x+\summ{y}{k-1}{k-1},\xi)\text{\dj} \by_{k-1}\text{\dj} \be_{k-1}.
\end{align}
Then there exists a constant $C>0$ dependent on $\sum_{j=1}^{k-1} |m_j|$ but independent of $k$ such that
\begin{align}\label{qthetaresult}
|b_{\theta}(x,\xi)| \lv C^{k+1} |a|_{l,l'}^{(\bm_k)} \jb{\xi}^{\overline{\bm}_k^{k}}.
\end{align} 
where 
\begin{align}
l=2\left\lceil \frac{d}{2}+1\right\rceil \quad\text{and}\quad l'=2\left\lceil \frac{d+\sum_{j=1}^{k-1}|m_j|}{2}+1\right\rceil.
\end{align}
\end{theorem}
\begin{proof}
For $d_0 \in \N$ and $1\leq j\leq k-1$ note that
\begin{align}
e^{-iy^j\cdot \eta^j}=\big(1+|y^j|^{2d_0}\big)^{-1}\big(1+(-\Delta_{\eta^j})^{d_0}\big)e^{-iy^j\cdot \eta^j}.
\end{align}
Choosing $d_0=\frac{l}{2}$, i.e. $2 d_0>d$, it is possible to repeatedly integrate by parts to obtain
\begin{align}
b_{\theta}(x,\xi)&=\Os\int_{\R^{2(k-1)d}} e^{-i\by_{k-1}\cdot \be_{k-1}}
\left[\prod_{j=1}^{k-1}\big(1+|y^j|^{2 d_0}\big)^{-1}\right]\\ \times
& \left[\prod_{j=1}^{k-1}\big(1+(-\Delta_{\eta^j})^{d_0}\big)\right]a(x,\xi+\theta \eta^1,\ldots,x+\summ{y}{k-1}{k-1},\xi)\text{\dj} \by_{k-1} \text{\dj} \be_{k-1}.
\end{align}
Note that
\begin{align}
\by_{k-1}\cdot\be_{k-1}=\sum_{j=1}^{k-1}\summ{y}{k-1}{j}\cdot (\eta^j-\eta^{j+1}) \quad\text{and}\quad y^j=\summ{y}{k-1}{j}-\summ{y}{k-1}{j}
\end{align}
with $\eta^k:=0$. Hence, we can now write
\begin{align}
\begin{split}
b_{\theta}(x,\xi)&=\Os \int_{\R^{2(k-1)d}} e^{-i\sum_{j=1}^{k-1}
\summ{y}{k-1}{j}\cdot (\eta^j-\eta^{j+1})}
\left[\prod_{j=1}^{k-1}\big(1+|\summ{y}{k-1}{j}-\summ{y}{k-1}{j-1}|^{2 d_0}\big)^{-1}\right]\\ \label{eqThetaExp1} \times
& \left[\prod_{j=1}^{k-1}\big(1+(-\Delta_{\eta^j})^{d_0}\big)\right]a(x,\xi+\theta \eta^1,\ldots,x+\summ{y}{k-1}{k-1},\xi)\text{\dj} \by_{k-1} \text{\dj} \be_{k-1}.
\end{split}
\end{align}
As before, it holds for some $k_j\in \N$:
\begin{align*}
 e^{-i\sum_{j=1}^{k-1}
\summ{y}{k-1}{j}\cdot (\eta^j-\eta^{j+1})}=|\eta^j-\eta^{j+1}|^{-2k_j} (-\Delta_{\summ{y}{k-1}{j}})^{k_j}  e^{-i\sum_{j=1}^{k-1}
\summ{y}{k-1}{j}\cdot (\eta^j-\eta^{j+1})}.
\end{align*}
Making in \eqref{eqThetaExp1} the change of variables
\begin{align}
\by_{k-1}=(y^1,\ldots,y^{k-1})\mapsto \overline{\overline{\by}}_{k-1}=(\summ{y}{k-1}{1},\ldots, \summ{y}{k-1}{k-1}),
\end{align}
and integrating by parts, we find for $0\leq k_j:=k_j(\eta^j,\eta^{j+1})\leq \frac{l'}{2}$,
%
%
\begin{align*}
&b_{\theta}(x,\xi)= \Os \int_{\R^{2(k-1)d}} e^{-i\sum_{j=1}^{k-1}
\overline{\by}_{k-1}^{j}\cdot (\eta^j-\eta^{j+1})}\left(\prod_{j=1}^{k-1}|\eta^j-\eta^{j+1}|^{-2k_j}\right)\\
&\times \left[\prod_{j=1}^{k-1} (-\Delta_{\overline{\by}_{k-1}^j})^{k_j}\right]r_{\theta}(x,\xi;\be_{k-1},
\overline{\overline{\by}}_{k-1})\text{\dj} \overline{\overline{\by}}_{k-1}\text{\dj} \be_{k-1},
\end{align*}
where 
\begin{align*}
r_{\theta}(x,\xi;\be_{k-1},\overline{\overline{\by}}_{k-1})
=&\left[\prod_{j=1}^{k-1} \big(1+ |\overline{\by}^j_{k-1}-
\summ{y}{k-1}{j-1}|^{2d_0}\big)^{-1}\right]
\left[\prod_{j=1}^{k-1}\big(1+(-\Delta_{\eta^j})^{d_0}\big)\right]\\
&\mspace{20mu}\times a(x,\xi+\theta \eta^1,\ldots,x+\summ{y}{k-1}{k-1},\xi)
\end{align*}
with $\overline{\by}^0_{k-1}:=0$. Note that $l'$ was chosen in such a way that the integrals exist even when taking into account the growth of the symbol in the integrands. Using the change of variables $\overline{\by}_{k-1}^{j}\mapsto w^j:=\overline{\by}_{k-1}^j-\overline{\by}_{k-1}^{j-1}$ we also find for some constant $C>0$ that
\begin{align}
\int_{\R^d}\big(1+|\overline{\by}_{k-1}^{j}-
\overline{\by}_{k-1}^{j-1}|^{2d_0}\big)^{-1}\text{\dj}\summ{y}{k-1}{j}\leq C.
\end{align}
Hence there exists a constant $C_1>0$ depending on $l$ such that
\begin{align}\label{qthetaestimate}
|b_{\theta}(x,\xi)|\lv& C_1^{k+1} |a|_{l,l'}^{(\bm_k)} \jb{\xi}^{m_k}\int_{\R^{(k-1)d}} 
\prod_{j=1}^{k-1} |\eta^j-\eta^{j+1}|^{-2k_j} \jb{\xi+\theta \eta^j}^{m_j}\text{\dj} 
\be_{k-1}.
\end{align}
The multiple symbol $a(x,\xi+\theta\eta^1,\ldots,x+\overline{\by}_{k-1}^{k-1},\xi)$ has $k-1$ arguments depending on $\eta^j$, $1\leq j\leq k-1$. The factor $\jb{\xi}^{m_k}$ corresponds to the growth of $a$ with respect to the last argument $\xi \in\R^d$. Next, set
\begin{align}
A_{j_0}:=\int_{\R^{j_0 d}}   \prod_{j=1}^{j_0} |\eta^j-\eta^{j+1}|^{-2k_j} \jb{\xi+\theta \eta^j}^{m_j}\text{\dj} 
\be_{j_0}.
\end{align}
The aim is to prove by induction that for there exists a constant $C_2>0$ independent of $j_0$ and $k$ such that
\begin{align}\label{Aj0_induction}
A_{j_0}\leq C_2^{j_0+1} \jb{\xi+\theta \eta^{j_0+1}}^{\overline{\bm}_k^{j_0}},
\end{align}
where $j_0=1,\ldots, k-1$ and $\eta^k=0$. Setting $j_0=k-1$ and using this estimate in \eqref{qthetaestimate} it holds that
\begin{align}
|b_{\theta}(x,\xi)|\lv C^k |a|_{l,l'}^{(\bm_k)} \jb{\xi}^{m_k} \jb{\xi+\theta\eta^k}^{\overline{\bm}_{k}^{k-1}}\lv C^k |a|_{l,l'}^{(\bm_k)} \jb{\xi}^{\overline{\bm}_k^k}
\end{align}
which is \eqref{qthetaresult}. For the induction, set
\begin{align}
&\Omega_{j,1}:=\{\eta^j\in \R^d:|\eta^j-\eta^{j+1}|\leq c_0\}\\
&\Omega_{j,2}:=\{\eta^j\in \R^d:c_0\leq |\eta^j-\eta^{j+1}|\leq c_0\jb{\xi+\theta\eta^{j+1}}\}\\
&\Omega_{j,3}:=\{\eta^j\in \R^d: |\eta^j-\eta^{j+1}|\geq c_0\jb{\xi+\theta\eta^{j+1}}\}
\end{align}
as well as
\begin{align}
k_j:=\begin{cases}
 0 &,\eta^j\in \Omega_{j,1},\\
\frac{l'}{2} &,\eta^j\in \Omega_{j,2}\cup \Omega_{j,3}.
\end{cases}
\end{align}
Assuming that \eqref{Aj0_induction} is true for $j_0-1$, it follows:
\begin{align}
A_{j_0}&=\intrd A_{j_0-1} |\eta^{j_0}-\eta^{j_0+1}|^{-2k_{j_0}} \jb{\xi+\theta \eta^{j_0}}^{m_{j_0}}\text{\dj} \eta^{j_0}\\
&\leq C_2^{j_0} \intrd |\eta^{j_0}-\eta^{j_0+1}|^{-2k_{j_0}} \jb{\xi+\theta \eta^{j_0}}^{\overline{\bm}_k^{j_0-1}} \jb{\xi+\theta\eta^{j_0}}^{m_{j_0}}\text{\dj} \eta^{j_0}\\
&=C_2^{j_0}\int_{\Omega_{j_0,1}\cup \Omega_{j_0,2}\cup \Omega_{j_0,3}} |\eta^{j_0}-\eta^{j_0+1}|^{-2k_{j_0}} \jb{\xi+\theta\eta^{j_0}}^{\overline{\bm}_k^{j_0}}\text{\dj} \eta^{j_0}.
\end{align}
On $\Omega_{j,1}\cup \Omega_{j,2}$ we have
$$
\theta|\eta^j-\eta^{j+1}|\leq c_0 \jb{\xi+\theta \eta^{j+1}},
$$
and thus by Lemma \ref{lemmaAux1} 
\begin{align}\label{etaJEstimate}
\frac{1}{2}\jb{\xi+\theta\eta^{j+1}}\leq \jb{\xi+\theta\eta^j}\leq 2\jb{\xi+\theta\eta^{j+1}}.
\end{align}
Looking at the proof of Lemma \ref{lemmaAux1} it furthermore holds 
\begin{align}
\jb{\xi+\theta\eta^j}-\jb{\xi+\theta\eta^{j+1}}\lv |\eta^j-\eta^{j+1}| 
\end{align}
i.e. on $\Omega_{j,3}$,
\begin{align}\label{etaJEstimate1}
\jb{\xi+\theta\eta^j}\lv |\eta^j-\eta^{j+1}|.
\end{align}
From \eqref{etaJEstimate} and \eqref{etaJEstimate1} we find
\begin{align}
&\int_{\Omega_{j_0,1}} |\eta^{j_0}-\eta^{j_0+1}|^{-2k_{j_0}} \jb{\xi+\theta\eta^{j_0}}^{\overline{\bm}_k^{j_0}} \text{\dj} \eta^{j_0}\\
\leq\:& 2^{\sum_{j=1}^{j_0}|m_j|}
\jb{\xi+\theta\eta^{j_0+1}}^{\overline{\bm}_k^{j_0}} \int_{\Omega_{j_0,1}} \text{\dj} \eta^{j_0}\\
\lv\:& 2^{\sum_{j=1}^{k-1}|m_j|}
\jb{\xi+\theta\eta^{j_0+1}}^{\overline{\bm}_k^{j_0}}
\end{align}
as well as
\begin{align}
&\int_{\Omega_{j_0,2}} |\eta^{j_0}-\eta^{j_0+1}|^{-2k_{j_0}} \jb{\xi+\theta\eta^{j_0}}^{\overline{\bm}_k^{j_0}} \text{\dj} \eta^{j_0}\\
\leq\:& 2^{\sum_{j=1}^{j_0}|m_j|} \jb{\xi+\theta\eta^{j_0+1}}^{\overline{\bm}_k^{j_0}} \int_{\Omega_{j_0,2}} |\eta^{j_0}-\eta^{j_0+1}|^{-l'}\text{\dj} \eta^{j_0}\\
\leq\:& 2^{\sum_{j=1}^{j_0}|m_j|} \jb{\xi+\theta\eta^{j_0+1}}^{\overline{\bm}_k^{j_0}} \int_{\Omega_{j_0,2}\cup \Omega_{j_0,3}} |\eta^{j_0}-\eta^{j_0+1}|^{-l'}\text{\dj} \eta^{j_0}\\
\lv\:& 2^{\sum_{j=1}^{k-1}|m_j|} \jb{\xi+\theta\eta^{j_0+1}}^{\overline{\bm}_k^{j_0}}
\end{align}
and
\begin{align}
&\int_{\Omega_{j_0,3}} |\eta^{j_0}-\eta^{j_0+1}|^{-2k_{j_0}} \jb{\xi+\theta\eta^{j_0}}^{\overline{\bm}_k^{j_0}} \text{\dj} \eta^{j_0}\\
\leq\:& C_3^{\sum_{j=1}^{j_0}|m_j|} \int_{\Omega_{j_0,3}}|\eta^{j_0}-\eta^{j_0+1}|^{-l'+(\overline{\bm}_k^{j_0})_+}\text{\dj} \eta^{j_0},\\
\lv\:& C_3^{\sum_{j=1}^{j_0}|m_j|} \jb{\xi+\theta \eta^{j_0+1}}^{-l'+(\overline{\bm}_k^{j_0})_+ + d},\\
\lv\:& C_3^{\sum_{j=1}^{k-1}|m_j|} \jb{\xi+\theta \eta^{j_0+1}}^{\overline{\bm}_k^{j_0}},
\end{align}
where $(\overline{\bm}_k^{j_0})_+=\max(\overline{\bm}_k^{j_0},0)$. Finally, choosing 
$$
C_2>2\cdot 2^{\sum_{j=1}^{k-1}|m_j|}+C_3^{\sum_{j=1}^{k-1}|m_j|}
$$
if follows
\begin{align}
A_{j_0}\leq& C_2^{j_0}\cdot \left(2\cdot 2^{\sum_{j=1}^{k-1}|m_j|}+C_3^{\sum_{j=1}^{k-1}|m_j|}\right)\jb{\xi+\theta\eta^{j_0+1}}^{\overline{\bm}_k^{j_0}}\\
\leq \:& C_2^{j_0+1}\jb{\xi+\theta\eta^{j_0+1}}^{\overline{\bm}_k^{j_0}}
\end{align}
which concludes the induction.
\end{proof}
The case $k=2$ in Theorem \ref{theoremThetaEstimate} is of special importance and is stated separately in the Corollary below (note that this also follows from the theory of double symbols).
\begin{corollary}\label{corollaryThetaSeminorm}
Let $a_1 \in S^{m_1,\psi}_{\rho_g}(\R^{2d})$ and 
$a_2\in S^{m_2,\psi}_{\rho_g}(\R^{2d})$, $g\in \{0,1,2\}$. Define $b_{\theta}$, $|\theta|\leq 1$,
by
\begin{equation}
b_{\theta}(x,\xi)=\Os \int_{\R^{2d}}e^{-iy\cdot\eta} a_1(x,\xi+\theta\eta)
a_2(x+y,\xi)\text{\dj} y\text{\dj} \eta.
\end{equation}
Then $\{b_{\theta}(x,\xi)\}_{|\theta|\leq 1}$ is a bounded set in
$S^{m_1+m_2,\psi}_{\rho_g}(\R^{2d})$. Furthermore, for any $l$, there exists
a constant $A_l'$ and an integer $l'$ independent of $\theta$ such
that
\begin{equation}
|b_{\theta}|_l^{(m_1+m_2)}\leq A_l'|a_1|_{l'}^{(m_1)}|a_2|_{l'}^{(m_2)}.
\end{equation}
\end{corollary}
\begin{proof}
This is Theorem \ref{theoremThetaEstimate} for $k=2$, $\bm_2=(m_1,m_2)$ and
$$
a(x,\xi+\theta\eta,x+y,\xi):=a_1(x,\xi+\theta \eta)a_2(x+y,\xi).
$$
\end{proof}
\begin{theorem}\label{kcompositionestimate}
Consider the symbol $a\in S^{\bm_k,\psi}_{\rho_g}$, $g\in \{0,1,2\}$, and define
\begin{align}\label{boscint}
\begin{split}
b(x,\xi):=&\Os \int_{\R^{2(k-1)d}} e^{-i\by_{k-1}\cdot \be_{k-1}}\\
&\mspace{20mu}\times a(x,\xi+\eta^1,x+ \summ{y}{k-1}{1},\ldots,\xi+\eta^{k-1},x+\summ{y}{k-1}{k-1},\xi)\text{\dj} \by_{k-1}\text{\dj} \be_{k-1}.
\end{split}
\end{align}
Then $b\in S^{\summ{m}{k}{k},\psi}_{\rho_g}$ and for any $l,l'\in \N$ there exists a constant $C>0$ such that
\begin{align}
|b|_{l,l'}^{(\summ{m}{k}{k})}\leq C^k |a|_{l_0,l_0'}^{(\bm_k)}.
\end{align}
\end{theorem}
\begin{proof}
Take $\chi(x,\xi)\in \mathcal{S}(\R^{2d})$ such that $\chi(0,0)=1$ and set
\begin{align}
\chi_{\varepsilon}(\by_k,\bxi_k):=\chi(\varepsilon y^1,\varepsilon \xi^1)\cdots \chi(\varepsilon y^k,\varepsilon \xi^k).
\end{align}
Using the definition of oscillatory integrals we can then write
\begin{align}
&a(x,D_x)u(x)\\
=\:&\Os \int_{\R^{2kd}} e^{-i\by_k\cdot \bxi_k} a(x,\xi^1,x+\overline{\by}_k^1,\ldots,\xi^k)u(x+\overline{\by}_k^k)\text{\dj} \by_k\text{\dj} \bxi_k\\
=\:&\lim_{\varepsilon\rightarrow 0}\int_{\R^{2kd}} e^{-i\by_k\cdot \bxi_k}\chi_{\varepsilon}(\by_k,\bxi_k) a(x,\xi^1,x+\overline{\by}_k^1,\ldots,\xi^k)u(x+\overline{\by}_k^k)\text{\dj} \by_k\text{\dj} \bxi_k.
\end{align}
By the change of variables $(y^k;\xi^1,\ldots,\xi^{k-1},\xi^k)\mapsto (x';\eta^1,\ldots,\eta^{k-1},\xi)$, where
\begin{align} (x';\eta^1,\ldots,\eta^{k-1},\xi):=(x+\overline{\by}_{k-1}^{k-1}+y^k;\xi^1-\xi^k,\ldots, \xi^{k-1}-\xi^k,\xi^k),
\end{align}
i.e. $\xi^j=\eta^j+\xi^k=\eta^j+\xi$, $1\leq j\leq k-1$ and $y^k=x'-x-\summ{y}{k-1}{k-1}$,
it follows
\begin{align}
&a(x,D)u(x)=\lim_{\varepsilon\rightarrow 0}\int_{\R^{2d}}e^{i(x-x')\cdot \xi}\\
&\times \Bigg[\iint_{\R^{2(k-1)d}} e^{-i\by_{k-1}\cdot \bxi_{k-1}}\chi_{\varepsilon}(y^1,\xi+\eta^1)\cdots \chi_{\varepsilon}(y^{k-1},\xi+\eta^{k-1}) \\
&\times \chi_{\varepsilon}(x'-x-\overline{\by}_{k-1}^{k-1},\xi) a(x,\xi+\eta^1,x+\overline{\by}_{k-1}^1,\ldots,\xi+\eta^{k-1},x+\overline{\by}_{k-1}^{k-1},\xi)\text{\dj} \by_{k-1}\text{\dj} \be_{k-1}\Bigg]\\
&\times u(x')\text{\dj} x' \text{\dj} \xi.
\end{align}
Comparing this expression with \eqref{boscint}, it is clear that $A(x,D_x)=B(x,D_x)$. Moreover, for any $\alpha,\beta \in \N_0^d$:
\begin{align}
&\partial_{\xi}^{\alpha}\partial_x^{\beta} b(x,\xi)=(2\pi)^{-(k-1)d}\Os \int_{\R^{2(k-1)d}} e^{-i \by_{k-1}\cdot \be_{k-1}}\\
&\times \partial_{\xi}^{\alpha}\partial_x^{\beta} a(x,\xi+\eta^1,\ldots,x+\overline{\by}_{k-1}^{k-1},\xi)\id \by_{k-1}\id \be_{k-1},
\end{align}
Elementary calculus yields that $\partial_{\xi}^{\alpha}\partial_x^{\beta}a$ can be expressed as a sum of $k^{|\alpha+\beta|}$ terms of symbols of class
$S^{\bm_k,\psi}_g(\R^{2kd})$. Hence, applying Theorem \ref{theoremThetaEstimate} with $\theta=1$
to $S^{\bm_k,\psi}_g(\R^{2kd})$, we obtain
\begin{align}
|\partial_{\xi}^{\alpha}\partial_x^{\beta} b(x,\xi)|\leq k^{|\alpha+\beta|} C^k |a|_{l_0,l_0'}^{(\bm_k)} \jb{\xi}^{\overline{\bm}_k}
\end{align}
for $l_0,l_0'\in\N$ large enough.
\end{proof}
\begin{theorem}\label{multiProduct}
Let $M>0$ be given and assume $\bm=(m_n)_{n\in  \N}$
is a sequence of real numbers satisfying
\begin{equation}
\sum_{n=1}^{\infty} |m_n|\leq M<\infty.
\end{equation}
For any $k\in\N$, $g\in \{0,1,2\}$ and $a_j\in S^{m_j,\psi}_{\rho_g}$, $1\leq j\leq k$, 
there exists a symbol
\begin{equation}
b_{k}\in S^{\overline{\bm}^k,\psi}_{\rho_g}(\R^{2d})
\end{equation}
where $\overline{\bm}^{k}:=\sum_{j=1}^{k} m_j$ such that
\begin{equation}
\begin{split}
b_{k}(x,D_x)=(a_1 \#_{\textsc{kn}} a_2\#_{\textsc{kn}}\cdots\#_{\textsc{kn}} a_k)(x,D_x):=a_1(x,D_x)a_2(x,D_x)\cdots a_k(x,D_x).
\end{split}
\end{equation}
Furthermore, for any $l\in \N_0$ there exist constants $C_l>0$
and $l'\in\N_0$ such that
\begin{equation}
|b_k|_l^{(\overline{\bm}^{k})}\leq C_l^{k}\prod_{j=1}^{k}
|a_j|_{l'}^{(m_j)}
\end{equation}
where $C_l$ and $l'$ depend only on $M$ and $l$.
\end{theorem}
\begin{proof}
This follows from Theorem \ref{kcompositionestimate} by noting that 
$$
b_{k}\in S^{\bm_k,\psi}_{\rho_g}(\R^{2kd})
$$
where $\bm_{k}=(m_1,\ldots,m_{k})$.
\end{proof}
The necessary theory to tackle the main problem of constructing a fundamental solution of \eqref{eqPDE} has now been introduced. In particular, we have extended the necessary statements from \cite{Kumanogo-75} concerning multiple symbols using weight functions to negative definite functions.

\section{Fundamental solution of a pseudo-differential evolution equation}\label{secFundamentalSolution}
In this section the fundamental solution of \eqref{eqPDE} is constructed. The proof of Theorem \ref{theoremFundamentalSolution} is given at the end. As before, let $\pi_{s,t}$ denote for $0\leq s\leq t\leq T$ an arbitrary partition of the interval $[s,t]$ into subintervals, cf. \eqref{eqPartition}. Define for a function $a:[0,T]\times \R^{2d}\rightarrow \C$ that satisfies Assumptions \ref{mainAssumptions}
\begin{align}\label{eqDefinitionP}
p(s,t;x,\xi)=e^{-\int_s^t a(\tau;x,\xi)\id \tau}
\end{align}
and
\begin{align}\label{eqMainComposition}
p(\pi_{s,t})=p(t_0,t_1)\#_{\textsc{kn}}p(t_1,t_2)\#_{\textsc{kn}}\ldots \#_{\textsc{kn}}p(t_k,t_{k+1}).
\end{align}
Obviously, $p(\pi_{s,t})$ can be decomposed into a principal symbol, a term of lower order, and a remainder term. However, it is crucial to understand how these terms depend on the number $k$ of partitions 
of the interval $[s,t]$. Using Fujiwara's method, cf. Appendix \ref{secFujiwara} and \cite{Fujiwara-91}, it is possible to find an expression for the remainder term of $p(\pi_{s,t};x,\xi)$ that can be estimated in terms of $t_{k+1}-t_0$ only, i.e. the length of the interval $[s,t]$.

Denote by $q_0(\pi_{t_0,t_{k+1}})$ the principal term of \eqref{eqMainComposition} and set
\begin{align}\label{eqDefinitionq1}
q_1(\pi_{t_0,t_{k+1}};x,\xi)&=\sum_{\wt{\alpha}_k\in \N_0^{kd},|\wt{\alpha}_k|=1}
D_x^{\alpha^k}p(t_k,t_{k+1};x,\xi)\\
&\mspace{20mu}\times \partial_{\xi}^{\alpha^k}
\Bigg(D_x^{\alpha_{k-1}}p(t_{k-1},t_k;x,\xi)\partial_{\xi}^{\alpha^{k-1}}\bigg(\cdots
D_x^{\alpha^2}p(t_2,t_3;x,\xi)\\
&\mspace{20mu}\times \partial_{\xi}^{\alpha^2}
\Big(D_x^{\alpha^1}p(t_1,t_2;x,\xi)\partial_{\xi}^{\alpha^1}
\big(p(t_0,t_1;x,\xi)\big)\Big)\cdots \bigg)\Bigg)
\end{align}
as well as
\begin{align}\label{eqDefinitionr}
r(\pi_{t_0,t_{k+1}};x,\xi)&=\sum_{\wt{\alpha}_k\in \N_0^{kd},|\ba_k|=2, |\alpha^k|\neq 0}\frac{|\alpha^k|}{\ba_k!}\\
&\msL\times\int_0^1 (1-\theta)^{|\alpha^k|-1}\Os\int_{\R^{2d}}e^{-iy\cdot\eta}
D_x^{\alpha^k}p(t_k,t_{k+1};x+y,\xi)\\
&\msL\times\partial_{\xi}^{\alpha^k}\Bigg(D_x^{\alpha^{k-1}}(t_{k-1},t_k;
x,\xi+\theta\eta)\partial_{\xi}^{\alpha^{k-1}}\bigg(\cdots D_x^{\alpha^2}p(t_2,t_3;x,\xi+\theta \eta)\\
&\msL\times \partial_{\xi}^{\alpha^2}\Big(D_x^{\alpha^1}(t_1,t_2;x,\xi+\theta \eta)
\partial_{\xi}^{\alpha^1}\big(p(t_0,t_1;x,\xi+\theta\eta)\big)\Big)
\cdots \bigg)\Bigg)\id y\id \eta\id \theta.
\end{align}
Also, define,
\begin{align}
(q_0+q_1)(\pi_{t_j,t_{j+1}})=p(t_j,t_{j+1}).
\end{align}
The restriction to derivatives of order upto $2$ in \eqref{eqDefinitionr} follows from the use of the cut-off function $\rho_2$ in Assumptions \ref{mainAssumptions}. Higher-order derivatives would therefore not further improve the decay of $r(\pi_{t_0,t_{k+1}};x,\xi)$. 

First we investigate what happens if we increase in \eqref{eqMainComposition} the number of compositions of symbols. Only the behavior of the principal symbol $q_0$ and the symbol $q_1$ is considered. 
\begin{lemma}\label{lemmaKey}
Let $a:[0,T]\times \R^{2d}\rightarrow \C$ satisfy Assumptions \ref{mainAssumptions}.
Then for any partition $\pi_{t_0,t_{k+1}}$, $k\in\N$, it follows that
\begin{align}\label{keyLemmaRes1}
(q_0+q_1)(\pi_{t_0,t_k})\#_{\textsc{kn}}p(t_k,t_{k+1})=(q_0+q_1+r)(\pi_{t_0,t_{k+1}}).
\end{align}
Furthermore there exist constants $l\in \N_0$, $b_l,c_l,d_l,e_l>0$ such that
\begin{align}\label{keyLemmaRes2}
  &\big|(q_0+q_1)(\pi_{t_0,t_k})\big|_l^{(0)}\lv b_l,
  &\intertext{and}\label{keyLemmaRes3}
  &\big|(q_0+q_1)(\pi_{t_0,t_{k+1}})-p(t_0,t_{k+1})\big|_l^{(2m)}\lv c_l(t_{k+1}-t_0)^2,
 \intertext{as well as} 
\label{keyLemmaRes4} &\big|r(\pi_{t_0,t_{k+1}})\big|_l^{(0)}\lv d_l (t_{k+1}-t_k),\\
\label{keyLemmaRes5} &\big|r(\pi_{t_0,t_{k+1}})\big|_l^{(m)}\lv e_l (t_k-t_0) (t_{k+1}-t_k).
\end{align}
for any partition $\pi_{t_0,t_{k+1}}$ and any $k\in \N$.
\end{lemma}
\begin{proof}
In order to see \eqref{keyLemmaRes1}, note that using a Taylor expansion it follows that the symbol $\big((q_0+q_1)(\pi_{t_0,t_k})\#_{\textsc{kn}}p(t_k,t_{k+1})\big)(x,\xi)$ is given by 
\begin{align*}
&\sum_{|\alpha^k|<2} \frac{1}{\alpha^k!} \Os\int_{\R^{2d}}e^{-iy\cdot \eta} D_x^{\alpha^k} p(t_k,t_{k+1};x+y,\xi)\\
&\qquad\times\partial_{\xi}^{\alpha^k}(q_0+q_1)(\pi_{t_0,t_k};x,\xi)\id y\id\eta\\
&+\sum_{|\alpha^k|=2} \frac{|\alpha^k|}{\alpha^k!}\int_0^1 (1-\theta)^{|\alpha^k|-1}\Os \int_{\R^{2d}} e^{-iy\cdot \eta} D_x^{\alpha^k}p(t_k,t_{k+1};x+y,\xi)\\
&\qquad\partial_{\xi}^{\alpha^k}(q_0+q_1)(\pi_{t_0,t_k};x,\xi+\theta\eta)\id y\id \eta\id \theta\\
=\:& (q_0+q_1)(\pi_{t_0,t_{k+1}};x,\xi)+r(\pi_{t_0,t_{k+1}};x,\xi).
\end{align*}
Note that
\begin{align}
q_0(\pi_{t_0,t_k};x,\xi)=e^{-\int_{t_0}^{t_k}a(\tau;x,\xi)}
\end{align}
and hence, cf. Appendix \ref{secEstimates},
\begin{align}\label{keyLemmaEQ3}
\left|\partial_{\xi}^{\alpha}D_x^{\beta} q_0(\pi_{t_0,t_k};x,\xi)\right|\lv_{\alpha,\beta} \begin{cases}
\jb{\xi}^{-\rho_2(|\alpha|)},\\
(t_k-t_0)\jb{\xi}^{m-\rho_2(|\alpha|)} &,|\alpha+\beta|\geq 1.
\end{cases}
\end{align}
It can also be shown that, cf. Appendix \ref{secEstimates},

\begin{align}\label{keyLemmaEQ4}
\left|\partial_{\xi}^{\alpha}D_x^{\beta}q_1(\pi_{t_0,t_{k+1}};x,\xi)\right|\lv_{\alpha,\beta} \begin{cases}
\jb{\xi}^{-1-\rho_1(|\alpha|)},\\
(t_{k+1}-t_0) \jb{\xi}^{m-1-\rho_1(|\alpha|)}.
\end{cases}
\end{align}
 By \eqref{keyLemmaEQ3} and \eqref{keyLemmaEQ4} it follows that \eqref{keyLemmaRes2} holds. Next, note that $q_0(\pi_{t_0,t_{k+1}};x,\xi)-p(t_0,t_{k+1};x,\xi)=0$, and hence
by \eqref{keyLemmaEQ4}, 
\begin{align}
|(q_0+q_1)(\pi_{t_0,t_{k+1}})-p(t_0,t_{k+1})|^{(2m)}_l\lv |q_1(\pi_{t_0,t_{k+1}})|_l^{(2m)}\lv_l (t_{k+1}-t_0)^2,
\end{align}
i.e. \eqref{keyLemmaRes3}. Finally, when rewriting $r(\pi_{t_0,t_{k+2}};x,\xi)$, it follows that
%
\begin{align}
r(\pi_{t_0,t_{k+2}};x,\xi)&=\sum_{|\alpha^{k+1}|=1} \frac{|\alpha^{k+1}|}{\alpha^{k+1}!} \int_0^1 (1-\theta)^{|\alpha^{k+1}|-1}\times\\
& \Os \iint e^{-iy\cdot \eta}\partial_{\xi}^{\alpha^{k+1}} q_1(\pi_{t_0,t_{k+1}};x,\xi+\theta\eta)\times \\
&D_x^{\alpha^{k+1}} q_0(\pi_{t_{k+1},t_{k+2}};x+y,\xi)\mathrm{d}y\mathrm{\eta}\mathrm{d}\theta\\
&+\sum_{|\alpha^{k+1}|=2} \frac{|\alpha^{k+1}|}{\alpha^{k+1}!} \int_0^1 (1-\theta)^{|\alpha^{k+1}|-1}\\
&\times \Os\iint e^{-iy\cdot \eta}\partial_{\xi}^{\alpha^{k+1}}q_0(\pi_{t_0,t_{k+1}};x,\xi+\theta\eta)\\
&\times D_x^{\alpha^{k+1}}q_0(\pi_{t_{k+1},t_{k+2}};x+y,\xi)\mathrm{d}y\mathrm{d}\eta \mathrm{d}\theta.
\end{align} 
By \eqref{keyLemmaEQ3} and \eqref{keyLemmaEQ4} and Corollary \ref{corollaryThetaSeminorm} as well as Lemma \ref{lemmaPeetre}, we get
\begin{align}
|r(\pi_{t_0,t_{k+2}})|_l^{(0)}\lv d_l (t_{k+2}-t_{k+1})
\end{align}
as well as
\begin{align}
|r(\pi_{t_0,t_{k+2}})|_l^{(m)}\lv e_l (t_{k+1}-t_0)(t_{k+2}-t_{k+1}).
\end{align}
\end{proof}
For the motivation of the following Lemma we refer to Appendix \ref{secFujiwara}. 
\begin{lemma}\label{lemmaFujiSkip}
Let $a:[0,T]\rightarrow \C$ satisfy Assumptions \ref{mainAssumptions}. Define for $k\in \N$ the symbol  $R(\pi_{t_0,t_{k+1}})$ by
\begin{align}\label{skipRes1}
p(t_0,t_1)\#_{\textsc{kn}} \ldots \# p(t_k,t_{k+1})
=(q_0+q_1)(\pi_{t_0,t_{k+1}})+R(\pi_{t_0,t_{k+1}}).
\end{align}
Then it follows that
\begin{align}\label{skipRes2}
R(\pi_{t_0,t_{k+1}})=\overset{\prime}{\sum}r(\pi_{t_0,t_{j_1}})\#_{\textsc{kn}}
r(\pi_{t_{j_1},t_{j_2}})\#_{\textsc{kn}}\cdots \#_{\textsc{kn}}r(\pi_{t_{j_{J-1}},t_{j_J}})\#_{\textsc{kn}}(q_0+q_1)(\pi_{t_{j_J},t_{k+1}}),
\end{align}
where $\overset{\prime}{\sum}$ stands for the summation with respect to all 
sequences of integers $(j_1,j_2,\ldots, j_J)$ with the property
\begin{align}
0<j_1-1<j_1<j_2-1<j_2<\cdots<j_{J-1}<j_{J}-1<j_J\leq k+1,
\end{align}
and, in the special case of $j_J=k$+1, we set $(q_0+q_1)(\pi_{t_{j_J},t_{k+1}};x,\xi)=1$.
Furthermore,
\begin{align}\label{skipRes3}
|R(\pi_{t_0,t_{k+1}})|_l^{(2m)}\lv_l (t_0-t_{k+1})^2,
\end{align}
for any partition $\pi_{t_0,t_{k+1}}$ and $k\in \N$.
\end{lemma}
\begin{proof}
Using induction, it follows by \eqref{keyLemmaRes1} that
\begin{align}\label{eqKeyLemma1}
&p(t_0,t_1)\#_{\textsc{kn}} \ldots \#_{\textsc{kn}} p(t_{k-1},t_k)\#_{\textsc{kn}} p(t_k,t_{k+1})\\
=&\big((q_0+q_1)(\pi_{t_0,t_k})+R(\pi_{t_0,t_k})\big)\#_{\textsc{kn}} p(t_k,t_{k+1})\\
=&(q_0+q_1)(\pi_{t_0,t_{k+1}})+r(\pi_{t_0,t_{k+1}})
+R(\pi_{t_0,t_k})\comp p(t_k,t_{k+1})\\
=&(q_0+q_1)(\pi_{t_0,t_{k+1}})+R(\pi_{t_0,t_{k+1}}).
\end{align}
For details concerning the last equality, cf. Appendix \ref{secSkip}.
%
%
%
%
Therefore, by \eqref{keyLemmaRes2}, \eqref{keyLemmaRes4}, \eqref{keyLemmaRes5} and Theorem \ref{multiProduct}, 
\begin{align}\label{skipEQ1}
&\left|R(\pi_{t_0,t_{k+1}})\right|_l^{(2m)}\\
\lv & C_{l}^J \overset{\prime}{\sum} |r(\pi_{t_0,t_{j_1}})|_{l'}^{(m)} |r(\pi_{t_{j_1},t_{j_2}})|_{l'}^{(m)}|r(\pi_{t_{j_2},t_{j_3}})|_{l'}^{(0)}\cdots |r(\pi_{t_{j_{J-1}},t_{j_J}})|_{l'}^{(0)} |(q_0+q_1)(\pi_{t_{j_J},t_{k+1}})|_{l'}^{(0)}\\
\lv_{l,l'} & \left(\prod_{n=0}^k \big(1+ c (t_{k+1}-t_0)(t_{n+1}-t_n)\big)\right)-1\\
\lv_{l,l'} & (t_{k+1}-t_0)^2,
\end{align}
which is \eqref{skipRes3}. For details, cf. Appendix \ref{secSkip}.
%
%
%
%
\end{proof}
The following Corollary combines the previous estimates and contains a crucial result that is needed to obtain Theorem \ref{theoremFundamentalSolution}. The following notion is used: given a partition $\pi_{s,t}$ of an interval $[s,t]$, we say that $\pi_{s,t}'$ is an arbitrary refinement of $\pi_{s,t}$ if
\begin{align}
\pi_{s,t}'=\begin{cases}
&t_0=t_{0,0}\leq t_{0,1}\leq \ldots \leq t_{0,j_0}=t_1,\\
&t_1=t_{1,0}\leq t_{1,1}\leq \ldots \leq t_{1,j_1}=t_2,\\
&\vdots\\
&t_k=t_{k,0}\leq t_{k,1}\leq \ldots \leq t_{k,j_k}=t_{k+1},
\end{cases}
\end{align}
for $j_k>0$, $j=0,1,2,\ldots, k$.
\begin{corollary}\label{corollaryMain}
Given a function $a:[0,T]:\R^{2d}\rightarrow \C$ that satisfies Assumptions \ref{mainAssumptions} there exists a symbol $p(\pi_{s,t})\in S^{0,\psi}_{\rho_2}(\R^{2d})$, such that
\begin{equation}\label{pPiDefEQ}
\begin{split}
p(\pi_{s,t})=e^{-\int_{s}^{t_1}a(\tau)\id \tau}\#_{\textsc{kn}}e^{-\int_{t_1}^{t_2}a(\tau)\id \tau}\#_{\textsc{kn}}\ldots\#_{\textsc{kn}} e^{-\int_{t_k}^t a(\tau)\id \tau}.
\end{split}
\end{equation}
Furthermore there exist constants $b_l,c_l>0$ and an integer $l$, such that
\begin{align}\label{pPiSemiEQ1}
|p(\pi_{s,t})|_l^{(0)}\lv b_l,
\end{align}
as well as
\begin{align}\label{pPiSemiEQ2}
|p(\pi_{s,t})-p(\pi_{s,t}')|_l^{(2m)}\lv c_l |\pi_{s,t}|,
\end{align}
where $\pi_{s,t}'$ is an arbitrary refinement of $\pi_{s,t}$.
\end{corollary}
\begin{proof}
Let us define $p(\pi_{s,t})$ by
\begin{align}\label{pPiSymbolEQ}
p(\pi_{s,t}):=(q_0+q_1)(\pi_{s,t})+R(\pi_{s,t}).
\end{align}
Then \eqref{pPiDefEQ} follows by Lemma \ref{lemmaFujiSkip}. By \eqref{keyLemmaRes2} and \eqref{skipRes3}, 
we get \eqref{pPiSemiEQ1}. Using \eqref{pPiSymbolEQ} it follows
\begin{align}
p(\pi_{t_j,t_{j+1}}')-p(t_j,t_{j+1})=\Big( (q_0+q_1)(\pi_{t_j,t_{j+1}}')-p(t_j,t_{j+1})\Big)+R(\pi_{t_j,t_{j+1}}').
\end{align}
By \eqref{keyLemmaRes3}, we now find that 
\begin{align}\label{pPiEQ1}
&|p(t_j,t_{j+1})-p(\pi_{t_{j},t_{j+1}}')|_l^{(2m)}\\
\lv_l& \big|(q_0+q_1)(\pi_{t_j,t_{j+1}}')-p(t_j,t_{j+1})\big|_l^{(2m)}+|R(\pi_{t_j,t_{j+1}})|_l^{(2m)}\\
\lv_l& (t_{j+1}-t_{j})^2.
\end{align}
As
\begin{align}
p(\pi_{s,t}')=p(\pi_{t_0,t_1}')\#_{\textsc{kn}} p(\pi_{t_1,t_2}')\#_{\textsc{kn}}\ldots\#_{\textsc{kn}} p(\pi_{t_k,t_{k+1}}')
\end{align}
%
we have for some $0\leq j\leq k$,
\begin{align}
&p(\pi_{t_0,t_{j-1}}')\#_{\textsc{kn}} p(\pi_{t_{j-1},t_{k+1}})-p(\pi_{t_0,t_j}')\#_{\textsc{kn}} p(\pi_{t_j,t_{k+1}})\\
=& p(\pi_{t_0,t_{j-1}}')\#_{\textsc{kn}} \big(p(\pi_{t_{j-1},t_j})-p(\pi_{t_{j-1},t_j}')\big)\#_{\textsc{kn}}p(\pi_{t_j,t_{k+1}}).
\end{align}
Hence
\begin{align}
p(\pi_{s,t})-p(\pi_{s,t}')= \sum_{j=0}^k p(\pi_{t_0,t_j}')\#_{\textsc{kn}}\big(p(t_j,t_{j+1})-p(\pi_{t_j,t_{j+1}}')\big)\#_{\textsc{kn}}  p(\pi_{t_{j+1},t_{k+1}}),
\end{align}
and thus by Theorem \ref{multiProduct}, \eqref{pPiEQ1} and \eqref{pPiSemiEQ1}, 
\begin{align}
|p(\pi_{s,t})-p(\pi_{s,t}')|_l^{(2m)}\lv_l & \sum_{j=0}^k |p(\pi_{t_0,t_j}')|_l^{(0)}\cdot |p(t_j,t_{j+1})-p(\pi_{t_j,t_{j+1}}')|_l^{(2m)}\cdot
|p(\pi_{t_{j+1},t_{k+1}})|_l^{(0)}\\
\lv_l & |\pi_{s,t}|\cdot \sum_{j=0}^k (t_{j+1}-t_j) \lv_l |\pi_{s,t}|
\end{align}
This proves \eqref{pPiSemiEQ2}.
\end{proof}
\begin{corollary}\label{corollaryPLim}
Let $a:[0,T]\times \R^{2d}\rightarrow \C$ be a function that satisfies Assumptions \ref{mainAssumptions}. Then there exists a symbol $p^{\mathrm{lim}}(s,t)\in S^{0,\psi}_{\rho_2}(\R^{2d})$ given by 
\begin{align}\label{pstardefinition}
p^{\lim}(s,t;x,\xi)=\lim_{|\pi_{s,t}|\rightarrow 0} \Os \int_{\R^{2kd}}
e^{-i \wt{y}_k\cdot \wt{\eta}_k - \sum_{j=0}^k\int_{t_j}^{t_{j+1}} a(\tau;x+\summ{y}{k}{j},\xi+\eta^{j+1})}\text{\dj} \wt{y}_k\text{\dj}\wt{\eta}_k, \qquad \eta^{k+1}=0,
\end{align}
such that $p(\pi_{s,t};x,\xi)$ converges to $p^{\mathrm{lim}}(s,t;x,\xi)$ in $S^{2m,\psi}_0$ as $|\pi_{s,t}|$ tends to $0$.
\end{corollary}
\begin{proof}
By Corollary \ref{corollaryMain}, \eqref{pPiSemiEQ1} and \eqref{pPiSemiEQ2}, there exists a symbol $p^{\lim}(s,t)\in S^{0,\psi}_{\rho_2}(\R^{2d})$ such that
for some constant $c_l>0$
\begin{align}
|p^{\lim}(s,t)|_l^{(0)}\leq c_l,
\end{align}
and
\begin{align}
|p(\pi_{s,t})-p^{\lim}(s,t)|_l^{(2m)}\lv c_l |\pi_{s,t}|.
\end{align}
\end{proof}
\begin{corollary}\label{corollaryFeynmanExists}
For a function $u\in L^2(\R^d)$ the \psd operator $p^{\lim}(s,t;x,D_x)$ satisfies
\begin{align}\label{eqPLimDefinition}
&p^{\lim}(s,t;x,D_x)u(x)\\
=\:&\lim_{|\pi_{s,t}|\rightarrow 0} e^{-\int_s^{t_1}a(\tau)\id \tau}(x,D_x)e^{-\int_{t_1}^{t_2}a(\tau)\id\tau}(x,D_x)\circ \ldots \circ e^{-\int_{t_k}^ta(\tau)\id\tau}(x,D_x)u(x)\\
=\:&\lim_{|\pi_{s,t}|\rightarrow 0} \Os\int_{\R^{2(k+1)d}}\mspace{-10mu}
e^{\sum_{j=0}^k i(x^{j+1}-x^{j})\cdot \xi^{j+1}-\int_{t_j}^{t_{j+1}}a(\tau;x^j,\xi^{j+1})\id \tau} u(x^{k+1})\text{\dj} \wt{x}_{k+1}\text{\dj} \wt{\xi}_{k+1},
\end{align}
in $L^2(\R^d)$ where $x^0\equiv x$.
\end{corollary}
\begin{proof}
This follows from Corollary \ref{corollaryPLim} and
\begin{align}
&\|p(\pi_{s,t};x,D_x)-p^{\lim}(s,t;x,D_x)\|_{L^2}\lv |p(\pi_{s,t})-p^{\lim}(s,t)|^{(2m)}_l \|u\|_{H^{2m,\psi}}\\
\lv\:& |\pi_{s,t}| \|u\|_{H^{2m,\psi}}.
\end{align}
\end{proof}
All the necessary tools and results for the proof of Theorem \ref{theoremFundamentalSolution} have been presented. It remains to make a few straightforward observations.

{\noindent \bf Proof of Theorem \ref{theoremFundamentalSolution}.} It remains to prove that $p^{\lim}(s,t;x,D_x)$ is the fundamental solution of \eqref{eqPDE}. Noting that $t_{k+1}=t$, a simple differentiation yields that
\begin{align}
\partial_t p^{\lim}(s,t)=-a(t)\#_{\textsc{kn}} p^{\lim}(s,t).
\end{align}
It is obvious that $p^{\lim}(s,s)=1$. 
\\
\qed
\\[0.8cm]
Concerning the validity of Corollary \ref{corollaryRoughSymbols} the following remarks can be made: the results so far have relied on the existence of semi-norms $|\cdot|_l^{(m)}$ of symbols for some $m\in \R$ and $l\in \N_0$. The index $l$ is usually chosen in such a way that the corresponding oscillatory integrals exist. Therefore, the constant $l$ depends on the dimension $d$ as well as on the order $m$ of the symbol. By choosing symbols that are $100d$-times differentiable we admit symbols with sufficient differentiability to work through all the proofs. From a probabilistic point of view, it would be desirable to work with symbols whose order of differentiability is independent of the dimension $d$, in particular with respect to the variable $\xi\in\R^d$. However, with the techniques and symbol classes used here, this is not possible. The differentiability conditions with respect to the variable $x\in\R^d$ of the symbol could be significantly improved by adapting techniques from \cite{Taylor-91}.

\section{Markov processes}\label{secMarkovProcesses}
In this section Markov processes are constructed using the fundamental solution 
\begin{align}
p^{\lim}(s,t;x,D_x)
\end{align}
 of \eqref{eqPDE} from Theorem \ref{theoremFundamentalSolution}. Markov processes constructed in this manner are typically time- and space-inhomogeneous. The aim is to prove Theorem \ref{theoremTransitionFunction}. This is done by checking that the conditions of a Markov evolution family are satisfied, cf. Definition \ref{definitionMarkovEvolutionFamily}. Once these properties are verified, the Markov process is constructed using standard theory.

{\noindent \bf Proof of Theorem \ref{theoremTransitionFunction}.} For the given symbol $a$ the translated symbol $a+R$, $R>0$ large enough, satisfies Assumptions \ref{mainAssumptions}. Therefore, by Corollary \ref{corollaryFeynmanExists}, the operator
\begin{align}
&p^{\lim}_R(s,t;x,D_x)u(x)\\
=\:&\lim_{|\pi_{s,t}|\rightarrow 0} e^{-\int_s^{t_1}\big(a(\tau)+R\big)\id \tau}(x,D_x)e^{-\int_{t_1}^{t_2}\big(a(\tau)+R\big)\id\tau}(x,D_x)\circ \ldots \circ e^{-\int_{t_k}^t\big(a(\tau)+R\big)\id\tau}(x,D_x)u(x)\\
=\:&e^{-R(t-s)}\lim_{|\pi_{s,t}|\rightarrow 0}e^{-\int_s^{t_1}a(\tau)\id \tau}(x,D_x)e^{-\int_{t_1}^{t_2}a(\tau)\id\tau}(x,D_x)\circ \ldots \circ e^{-\int_{t_k}^ta(\tau)\id\tau}(x,D_x)u(x)\\
=\:&e^{-R(t-s)}p^{\lim}(s,t;x,D_x)u(x)
\end{align}
is well-defined. Our first aim is show that the family of operators $p^{\lim}(s,t;x,D_x)$ satisfies the conditions set out in Definition \ref{definitionMarkovEvolutionFamily} on $C_{\infty}(\R^d)$. 

It is obvious that the operator $p^{\lim}(s,t;x,D_x)$ is a linear operator and that $p^{\lim}(s,s;x,D_x)=\Ident$. Moreover,
\begin{align}
p^{\lim}(s,r)\#_{\textsc{kn}}p^{\lim}(r,t)=p^{\lim}(s,t).
\end{align}
It remains to check the following three properties:
\begin{align}
&f\geq 0 \Rightarrow p^{\lim}(s,t;x,D_x)f\geq 0\:\text{ for all }\:0\leq s\leq t<\infty, f\in C_{\infty}(\R^d),\\
&p^{\lim}(s,t;x,D_x)\text{ is a contraction on $C_{\infty}(\R^d)$, i.e. }\|p^{\lim}(s,t;x,D_x)\|\leq 1 \:\text{ for }\:0\leq s\leq t<\infty,\\
&p^{\lim}(s,t;x,D_x)1=1\:\text{ for all }\:t\geq 0.
\end{align}
By \eqref{eqPLimDefinition} it suffices to show these properties for
\begin{align}
e^{-\int_{t_{j-1}}^{t_j}a(\tau)\id \tau}(x,D_x)
\end{align}
where $j\in \{1,\ldots,k+1\}$. Let $\pi_{t_{j-1},t_j}=\{t_{j-1}=t_{j-1,0}<t_{j-1,1}<\ldots<t_{j-1,j_1}=t_j\}$, $j_1>0$, be a partition of the interval $[t_{j-1},t_j]$. Also, define
\begin{align}
p_{x_0}(t_{j-1},t_j;\xi)=e^{-\int_{t_{j-1}}^{t_j}a(\tau;x_0,\xi)\id \tau}
\end{align}
as well as
\begin{align}
p_{t_{j-1,j-1},x_0}(t_{j-1,j-1},t_{j,j-1};\xi)=e^{-(t_{j-1,j}-t_{j-1,j-1})a(t_{j-1,j-1};x_0,\xi)}.
\end{align}
For every fixed $t_{j-1,j-1}\in \pi_{t_{j-1},t_j}$, $x_0\in \R^d$, $\xi\mapsto a(t_{j-1};x_0,\xi)$ is a continuous negative definite function with associated convolution semigroup
$(\mu_r^{t_{j-1,j-1},x_0})_{r\geq 0}$ such that
\begin{align}
p_{t_{j-1,j-1},x_0}(t_{j-1,j-1},t_{j-1,j};D_x)u(x)=\int_{\R^d}u(x-y)\mu_{t_{j-1,j}-t_{j-1,j-1}}^{t_{j-1,j-1},x_0}(\id y).
\end{align}
It is obvious that $p_{t_{j-1,j-1},x_0}(t_{j-1,j-1},t_{j-1,j};D_x)$ is positivity preserving and that
\begin{align}
\|p_{t_{j-1,j-1},x_0}(t_{j-1,j-1},t_{j-1,j};D_x)u\|_{\infty}\leq \|u\|_{\infty}.
\end{align}
Therefore the operator
\begin{align}
\prod_{j=1}^{j_1}p_{t_{j-1,j-1},x_0}(t_{j-1,j-1},t_{j-1,j};D_x)
\end{align}
is also positivity preserving as well as a contraction on $C_{\infty}(\R^d)$. Note that
\begin{align}
p_{x_0}(t_{j-1},t_j;\xi)=\lim_{|\pi_{t_{j-1},t_j}|\rightarrow 0}\prod_{j=1}^{j_1}p_{t_{j-1,j-1},x_0}(t_{j-1,j-1},t_{j,j-1};\xi)
\end{align}
and it follows that $p_{x_0}(t_{j-1},t_j;D_x)$ is positivity preserving and
\begin{align}
\|p_{x_0}(t_{j-1},t_j;D_x)u\|_{\infty}\leq \|u\|_{\infty}.
\end{align}
In conclusion
\begin{align}
\begin{split}
\left\|e^{-\int_{t_{j-1}}^{t_j}a(\tau)\id \tau}(x,D_x)u\right\|_{\infty}=&\sup_{x\in\R^d}\left|e^{-\int_{t_{j-1}}^{t_j}a(\tau)\id \tau}(x,D_x)u(x)\right|\\
=&\sup_{x_0\in \R^d}\|p_{x_0}(t_{j-1},t_j;D_x)u\|_{\infty}\\
\leq& \|u\|_{\infty}.
\end{split}
\end{align}
Next let $u(x_0)\geq 0$ for $x_0\in \R^d$. Then
\begin{align}
\left[e^{-\int_{t_{j-1}}^{t_j}a(\tau)\id \tau}(x,D_x)u\right](x_0)&=\int_{\R^d}e^{ix_0\cdot \xi}e^{-\int_{t_{j-1}}^{t_j}a(\tau;x_0,\xi)\id \tau}\wh{u}(\xi)\id \xi\\
&=p_{x_0}(t_{j-1},t_j;D_x)u(x_0)\geq 0.
\end{align}
In order to prove that $e^{-\int_{t_{j-1}}^{t_j}a(\tau)\id \tau}(x,D_x)1(x)=1(x)$, simply note that
\begin{align}
e^{-\int_{t_{j-1}}^{t_j}a(\tau)\id \tau}(x,D_x)1(x)=e^{ix\cdot 0}\left[e^{-\int_{t_{j-1}}^{t_j}a(\tau)\id \tau}(x,D_x)e^{-i(\cdot,0)}\right](x)=e^{-\int_{t_{j-1}}^{t_j}a(\tau;x,0)\id \tau}=1(x).
\end{align}
For $t\geq 0$ and $x\in\R^d$ fixed, the mapping 
\begin{align}
u\mapsto p^{\lim}(s,t;x,D_x)u
\end{align} 
is by the above a linear, continuous and positive operator on $C_{\infty}(\R^d)$. By a variant of Riesz' representation theorem it follows that there exists a unique Borel measure
$p_{s,t}(x,\id y)$ on $\mathcal{B}^d(\R^d)$ such that for $u\in C_{\infty}(\R^d)$
\begin{align}\label{eqPLimRepresentationMeasure}
p^{\lim}(s,t;x,D_x)u(x)=\int_{\R^d}u(y)p_{s,t}(x,\id y).
\end{align}
Therefore $p_{s,t}(x,\id y)$ is indeed a probability measure and the operator $p^{\lim}(s,t;x,D_x)$ can be extended to $B_b(\R^d)$ by defining an operator
\begin{align}
\wt{p}_{\lim}(s,t;x,D_x)u(x)=\int_{\R^d}u(y)p_{s,t}(x,\id y)
\end{align}
for all $u\in B_b(\R^d)$. Using \eqref{eqPLimRepresentationMeasure} it is clear that the operator $\wt{p}_{\lim}(s,t;x,D_x)$ satisfies all the properties of a Markov evolution family. The construction of a Markov process from a Markov evolution family of operators is standard, cf. \cite{Dynkin-65} for details.
\\
\qed

\section{Appendix}
\subsection{Fujiwara's representation of the remainder term}\label{secFujiwara}
The idea behind equations \eqref{skipRes1} and \eqref{skipRes2} is based on Fujiwara's representation \cite{Fujiwara-91} of the remainder term of a multiple composition of pseudo-differential operators. Let $a_1,\ldots,a_k \in S^{\wt{m}_k,\psi}_{\rho_g}$, $g\in \{0,1,2\}$, for some $k\in \N$. Then
\begin{align}\label{eqApp1eq1}
\begin{split}
&(a_1\comp \ldots \comp a_k)(x,\xi)\\
=&\Os\int_{\R^{2(k-1)d}} e^{i\wt{y}_{k-1}\cdot \wt{\eta}_{k-1}} 
a(x,\xi,y^1,\eta^1,\ldots,y^{k-1},\eta^{k-1})
\text{\dj} \wt{y}_{k-1}\text{\dj} \wt{\eta}_{k-1}
\end{split}
\end{align}
where
\begin{align}
&a(x,\xi,y^1,\eta^1,\ldots,y^{k-1},\eta^{k-1})\\
=&a_1(x,\xi+\eta^1)a_2(x+y^1,\xi+\eta^2)\cdots a_{k-1}(x+\summ{y}{k-1}{k-2},\xi+\eta^{k-1})a_k(x+\summ{y}{k-1}{k-1},\xi).
\end{align}
Integrate \eqref{eqApp1eq1} with respect to $(y^1,\eta^1)$ and find that
\begin{align}
&\Os\int_{\R^{2d}}e^{-i y^1\cdot \eta^1}a(x,\xi,y^1,\eta^1,\ldots,y^{k-1},\eta^{k-1})\text{\dj} y^1 \text{\dj} \eta^1\\
=&S_1a(x,\xi,y^2,\eta^2,\ldots,y^{k-1},\eta^{k-1})+R_1a(x,\xi,y^2,\eta^2,\ldots, y^{k-1},\eta^{k-1})
\end{align}
where $S_1a$ denotes the main part and $R_1a$ denotes the remainder:
\begin{align*}
&S_1a(x,\xi,y^2,\eta^2,\ldots,y^{k-1},\eta^{k-1})\\
=&\bigg[\sum_{|\alpha^1|<N}\frac{1}{\alpha^1!}D_x^{\alpha^1}a_2(x,\xi+\eta^2)\partial_{\xi}^{\alpha^1}a_1(x,\xi)\bigg]a_3(x+y^2,\xi+\eta^3)\cdots\\
&\times a_{k-1}(x+y^2+\ldots+y^{k-2},\xi+\eta^{k-1})a_k(x+y^2+\ldots+y^{k-1},\xi),
\end{align*}
and
\begin{align*}
&R_1a(x,\xi,y^2,\eta^2,\ldots,y^{k-1},\eta^{k-1})\\
=&\bigg[\sum_{|\alpha^1|=N}\frac{|\alpha^1|}{\alpha^1!} \int_0^1 (1-\theta)^{|\alpha^1|-1}\bigg(\Os\int_{\R^{2d}}e^{-iy^1\cdot \eta^1}D_x^{\alpha^1}a_2(x+y^1,\xi+\eta^2)\\
&\times \partial_{\xi}^{\alpha^1}a_1(x,\xi+\theta\eta^1)\text{\dj} y^1\text{\dj} \eta^1\bigg)\mathrm{d}\theta\bigg]a_3(x+y^2,\xi+\eta^3)\cdots\\
&\times a_{k-1}(x+y^2+\ldots+y^{k-2},\xi+\eta^{k-1})a_k(x+y^2+\ldots+y^{k-1},\xi).
\end{align*}
 Integrate $S_1a$ over $(y^2,\eta^2)$ and obtain
\begin{align}
&\Os\int_{\R^{2d}}e^{-iy^2\cdot \eta^2}S_1a(x,\xi,y^2,\eta^2,\ldots,y^{k-1},\eta^{k-1})\text{\dj} y^2\text{\dj} \eta^2\\
=&S_2S_1a(x,\xi,y^3,\eta^3,\ldots,y^{k-1},\eta^{k-1})+R_2S_1a(x,\xi,y^3,\eta^3,\ldots,y^{k-1},\eta^{k-1}).
\end{align}
Repeating this procedure we finally integrate over $(y^{k-1},\eta^{k-1})$ and find
\begin{align}
&\Os\int_{\R^{2d}} e^{-iy^{k-1}\cdot \eta^{k-1}} S_{k-2}\cdots S_1a(x,\xi,y^{k-1},\eta^{k-1})\text{\dj} y^{k-1}\text{\dj} \eta^{k-1}\\
=&S_{k-1}\cdots S_1a(x,\xi)+R_{k-1}S_{k-2}\cdots S_1a(x,\xi).
\end{align}
Next, consider the remainder terms. Integration of $R_1a$ with respect to $(y^2,\eta^2)$ is complicated, 
hence this is skipped and $R_1a$ is integrated with respect to $(y^3,\eta^3)$ beforehand:
\begin{align}
&\Os\int_{R^{2d}} e^{-iy^3\cdot \eta^3} R_1a(x,\xi,y^2,\eta^2,\ldots,y^{k-1},\eta^{k-1})\text{\dj} y^3\text{\dj} \eta^3\\
=& S_3R_1a(x,\xi,y^2,\eta^2,y^4,\eta^4,\ldots,y^{k-1}\eta^{k-1})+R_3R_1a(x,\xi,y^2,\eta^2,y^4,\eta^4,\ldots, y^{k-1},\eta^{k-1}).
\end{align}
Skip integration of $R_2S_1a$ with respect to $(y^3,\eta^3)$ and integrate with respect to $(y^4,\eta^4)$ first:
\begin{align}
&\Os\int_{\R^{2d}}e^{-iy^4\cdot \eta^4}R_2S_1a(x,\xi,y^3,\eta^3,\ldots,y^{k-1},\eta^{k-1})\text{\dj} y^4\text{\dj} \eta^4\\
=&S_4R_2S_1a(x,\xi,y^3,\eta^3,y^5,\eta^5,\ldots,y^{k-1},\eta^{k-1})+R_4R_2S_1a(x,\xi,y^3,\eta^3,y^5,\eta^5,\ldots,y^{k-1},\eta^{k-1}).
\end{align}
In the same manner, we integrate $S_3R_1a$ with respect to $(y^4,\eta^4)$ but do not integrate $R_3R_1a$. Repeating this procedure, we arrive at
\begin{align}\label{eqApp1eq2}
(a_1\comp \ldots \comp a_k)(x,\xi)= S_{k-1}\cdots S_1 a(x,\xi)+\overset{\prime}{\sum}a_{j_1,\ldots,j_{J}}(x,\xi)
\end{align}
where $\overset{\prime}{\sum}$ stands for the summation with respect to sequences $(j_1,\ldots j_J)$ of integers with the property
\begin{align}
0<j_1-1<j_1<j_2-1<j_2<\ldots<j_J-1<j_J\leq k-1<j_{J+1}=k
\end{align}
and
\begin{align}
a_{j_1,\ldots j_J}(x,\xi)=\Os\int_{R^{2Jd}} e^{-i\sum_{j\in\{j_1,\ldots j_J\}} y^j\cdot \eta^j}b_{j_1,\ldots,j_J}(x,\xi,y^{j_1},\eta^{j_1},\ldots y^{j_J},\eta^{j_J})\text{\dj} y^{j_1}\eta^{j_1}\ldots y^{j_J}\eta^{j_J}.
\end{align}
Here,
\begin{align}
b_{j_1,\ldots,j_J}(x,\xi,y^{j_1},\eta^{j_1},\ldots y^{j_J},\eta^{j_J})=(Q_{k-1}Q_{k-2}\cdots Q_1a)(x,\xi,y^{j_1},\eta^{j_1},\ldots y^{j_J},\eta^{j_J})
\end{align}
where
\begin{align}
Q_j=\begin{cases}
R_j, &\text{ if }j=j_1-1,j_2-1,\ldots, j_J-1,\\
\Ident, &\text{ if }j=j_1,\ldots,j_J,\\
S_j, &\text{ otherwise.}
\end{cases}
\end{align}

Now apply these ideas in the situation of \eqref{skipRes1}. Consider for ease of notation $k=4$, i.e.
 \begin{align}
&a(\pi_{t_0,t_5};x,\xi,y^1,\eta^1,y^2,\eta^2,y^3,\eta^3)\\
=&p(t_0,t_1;x,\xi+\eta^1)p(t_1,t_2;x+\summ{y}{3}{1},\eta^2)p(t_2,t_3;x+\summ{y}{3}{2},\xi+\eta^3)p(t_3,t_4;x+\summ{y}{3}{3},\xi).
\end{align}
By \eqref{eqApp1eq2} the following terms need to be calculated:
\begin{align*}
&S_3S_2S_1a=(q_0+q_1)(\pi_{t_0,t_4}),\\
&\int_{\R^{2d}}e^{-iy^2\cdot \eta^2}S_3R_1a(y^2,\eta^2)\text{\dj} y^2\text{\dj} \eta^2=r(\pi_{t_0,t_2})\comp (q_0+q_1)(\pi_{t_2,t_4}),\\
&\int_{\R^{2d}}e^{-iy^3\cdot \eta^3}R_2S_1a(y^3,\eta^3)\text{\dj} y^3\text{\dj} \eta^3=r(\pi_{t_0,t_3})\comp (q_0+q_1)(\pi_{t_3,t_4}),\\
&R_3S_2S_1a=r(\pi_{t_0,t_4}),\\
&\int_{\R^{2d}}e^{-iy^2\cdot\eta^2}R_3R_1a(y^2,\eta^2)\text{\dj} y^2\text{\dj} \eta^2=r(\pi_{t_0,t_2})\comp r(\pi_{t_2,t_4}).
\end{align*}
Therefore,
\begin{align}
&p(t_0,t_1)\comp p(t_1,t_2)\comp p(t_2,t_3)\comp p(t_3,t_4)\\
=&(q_0+q_1)(\pi_{t_0,t_4})+r(\pi_{t_0,t_4})+r(\pi_{t_0,t_3})\comp (q_0+q_1)(\pi_{t_3,t_4})+r(\pi_{t_0,t_2})\comp (q_0+q_1)(\pi_{t_2,t_4})\\
&+r(\pi_{t_0,t_2})\comp r(\pi_{t_2,t_4})\\
=&(q_0+q_1)(\pi_{t_0,t_4})+R(\pi_{t_0,t_4}).
\end{align}
In this particular case, the same result can be obtained by a straightforward application of \eqref{keyLemmaRes1}.


\subsection{Estimates for sequences that skip}\label{secSkip}
In this section some of the estimates from Lemma \ref{lemmaFujiSkip} are explained in more detail. Consider for a fixed $k\in\N$ the set $M_k$ of all sequences of integers
$(j_1,j_2,\ldots,j_J)$ with the property 
$$
0<j_1-1<j_1<j_2-1<j_2<\cdots<j_{J-1}<j_J-1<j_J\leq k.
$$
The set $M_{k+1}$ can now be constructed in the following way: start with the set $M_k$ and add an additional sequence containing just the number $k+1$. Then
add all the sequences of the set $M_{k-1}$ but append to every sequence in this set the number $k+1$ beforehand. This explains the last equality in \eqref{eqKeyLemma1}.

Consider next the inequality
\begin{align*}
&C_{l}^J \overset{\prime}{\sum} |r(\pi_{t_0,t_{j_1}})|_{l'}^{(m)} |r(\pi_{t_{j_1},t_{j_2}})|_{l'}^{(m)}|r(\pi_{t_{j_2},t_{j_3}})|_{l'}^{(0)}\cdots |r(\pi_{t_{j_{J-1}},t_{j_J}})|_{l'}^{(0)} |(q_0+q_1)(\pi_{t_{j_J},t_{k+1}})|_{l'}^{(0)}\\
\lv_{l,l'} & \left(\prod_{n=0}^k \big(1+ c (t_{k+1}-t_0)(t_{n+1}-t_n)\big)\right)-1.\\
\end{align*}
As indicated, use \eqref{keyLemmaRes5} on the first two seminorms of order $m$:
\begin{align*}
|r(\pi_{t_0,t_{j_1}})|_{l'}^{(m)}\leq e_l (t_{j_1-1}-t_0)(t_{j_1}-t_{j_1-1})\leq e_l (t_{k+1}-t_0)(t_{j_1}-t_{j_1-1}),\\
|r(\pi_{t_{j_1},t_{j_2}})|_{l'}^{(m)}\leq e_l (t_{j_2-1}-t_{j_1})(t_{j_2}-t_{j_2-1})\leq e_l (t_{k+1}-t_0)(t_{j_2}-t_{j_2-1}).
\end{align*}
For the final term, use \eqref{keyLemmaRes2}:
\begin{align*}
|(q_0+q_1)(\pi_{t_{j_J},t_{k+1}})|_{l'}^{(0)}\leq b_l.
\end{align*}
For the remaining terms, use \eqref{keyLemmaRes4}:
\begin{align*}
|r(\pi_{t_{j_i},t_{j_{i+1}}})|_{l'}^{(0)}\leq d_l
\begin{cases}
(t_{k+1}-t_0),\\
(t_{j_{i+1}}-t_{j_{i+1}-1}).
\end{cases}
\end{align*}

The final inequality 
\begin{align}
 \left(\prod_{n=0}^k \big(1+ c (t_{k+1}-t_0)(t_{n+1}-t_n)\big)\right)-1\lv_{l,l'} (t_{k+1}-t_0)^2\\
\end{align}
can be shown using induction, cf. \cite{Fujiwara-91}.

\subsection{Some Estimates for Negative Definite Symbols}\label{secEstimates}
The following tools are needed throughout this section, cf. \cite{NIST-10}, \cite{Constantine-96}. \\[0.3cm]
\emph{Fa{\`a} di Bruno formula:} for $\gamma\in\N_0^d$, $f:\R\rightarrow \R$ and $g:\R^{2d}\rightarrow \R$,
\begin{align}\label{eqFaa}
\partial_{x}^{\gamma}(f\circ g)(x)=\sum_{j=1}^{|\gamma|}f^{(j)}\big(g(x)\big)\sum_{\substack{\gamma^1+\ldots+\gamma^j=\gamma\\ \gamma^1,\ldots,\gamma^j\in \N_0^d}}c_{\gamma^1,\ldots,\gamma^j}\prod_{l=1}^j \partial_x^{\gamma^{l}}g(x).
\end{align}
\emph{Exponential function estimate:} for $j=0,1,\ldots$, $s\in (0,\infty)$, there exists a constant $A_j>0$ such that
\begin{align}\label{eqExp}
s^j e^{-s}\leq \left(\frac{j}{e}\right)^l=A_j<\infty.
\end{align}

First, observe that 
\begin{align}
|q_0(\pi_{t_0,t_{k+1}};x,\xi)|\leq C
\end{align}
for some constant $C>0$. For $\alpha,\beta\in\N_0^d$ such that $|\alpha+\beta|>0$, it follows using \eqref{eqFaa} that
\begin{align}
&\partial_{\xi}^{\alpha}D_x^{\beta}q_0(\pi_{t_0,t_{k+1}};x,\xi)\\
=&\sum_{j=1}^{|\alpha+\beta|}q_0(\pi_{t_0,t_{k+1}};x,\xi) \sum_{\substack{\alpha^1+\ldots+\alpha^j=\alpha, \alpha^1,\ldots,\alpha^j\in\N_0^d\\ \beta^1+\ldots+\beta^j=\beta, \beta^1,\ldots,\beta^j\in\N_0^d}}c_{\alpha^1,\beta^1,\ldots,\alpha^j,\beta^j}\prod_{l=1}^j \partial_{\xi}^{\alpha^l}D_x^{\beta^l}\left(-\int_{t_0}^{t_{k+1}} a(\tau;x,\xi)\id \tau\right).
\end{align}
Note that 
\begin{align}\label{eqApp31}
\partial_{\xi}^{\alpha^l}D_x^{\beta^l}\left(-\int_{t_0}^{t_{k+1}} a(\tau;x,\xi)\id \tau\right)=-\int_{t_0}^{t_{k+1}} \frac{\partial_{\xi}^{\alpha^l}D_x^{\beta^l}a(\tau;x,\xi)}{\Re a(\tau;x,\xi)}\Re a(\tau;x,\xi)\id \tau.
\end{align}
and it follows by Assumptions \ref{mainAssumptions} that
\begin{align}
\left|\frac{\partial_{\xi}^{\alpha^l}D_x^{\beta^l}a(\tau;x,\xi)}{\Re a(\tau;x,\xi)}\right|\lv_{\alpha^l,\beta^l} \jb{\xi}^{-\rho_2(|\alpha^l|)}.
\end{align}
Taking into account \eqref{eqExp} it holds that
\begin{align}
\left(\int_{t_0}^{t_{k+1}} \Re a(\tau;x,\xi)\id \tau\right)^j e^{-\int_{t_0}^{t_{k+1}} \Re a(\tau;x,\xi)\id \tau}\leq A_j.
\end{align}
Using in addition the subadditivity of the function $\rho_2$, it follows that 
\begin{align}
\prod_{l=1}^j \left|\partial_{\xi}^{\alpha^l}D_x^{\beta^l}\left(-\int_{t_0}^{t_{k+1}}a(\tau;x,\xi)\id \tau\right)\right|\lv_{\alpha^l,\beta^l}\jb{\xi}^{-\rho_2(|\alpha|)} \left(\int_{t_0}^{t_{k+1}} |\Re a(\tau;x,\xi)|\id \tau\right)^j.
\end{align}
In conclusion,
\begin{align}
&|\partial_{\xi}^{\alpha}D_x^{\beta}q_0(s,t;x,\xi)|\\
\lv_{\alpha,\beta}&\jb{\xi}^{-\rho_2(|\alpha|)}\sum_{j=1}^{|\alpha+\beta|}e^{-\int_{t_0}^{t_{k+1}} \Re a(\tau;x,\xi)\id \tau} \left(\int_{t_0}^{t_{k+1}} |\Re a(\tau;x,\xi)|\id \tau\right)^j\\
\lv_{\alpha,\beta}&\begin{cases}
\jb{\xi}^{-\rho_2(|\alpha|)}\\
(t_{k+1}-t_0)\jb{\xi}^{m-\rho_2(|\alpha|)}
\end{cases}
\end{align}
This proves \eqref{keyLemmaEQ3}. 

In order to prove \eqref{keyLemmaEQ4} note that $|\wt{\alpha}_k|=1$ implies that it sufficient to show the following estimate for all $\alpha,\alpha',\beta\in \N_0^d$, $|\alpha'|=1$:
\begin{align}
\left|\partial_{\xi}^{\alpha}D_x^{\beta}\partial_{\xi}^{\alpha'}e^{-\int_{t_{j-1}}^{t_j}a(\tau;x,\xi)\id \tau}\right|\lv_{\alpha,\beta}\begin{cases}
\jb{\xi}^{-1-\rho_1(|\alpha|)},\\
(t_j-t_{j-1})\jb{\xi}^{m-1-\rho_1(|\alpha|)}.
\end{cases}
\end{align}
But this follows in a straightforward manner from the calculations above.

\end{document}